\documentclass[reqno]{amsart}

\usepackage[backend=biber, style=abbrv, sorting=none, bibstyle=numeric, citestyle=numeric, sorting=nyt,maxbibnames=99, doi=false, url=false]{biblatex}

\renewbibmacro{in:}{}
\bibliography{references.bib}
\usepackage{amssymb}
\usepackage{calrsfs}
\usepackage{graphics,graphicx, mathrsfs}
\usepackage{enumerate}
\usepackage{enumitem}
\usepackage{url}
\usepackage{xcolor}
\definecolor{vio}{rgb}{0.54, 0.17, 0.89}
\usepackage[ruled, lined, linesnumbered, longend]{algorithm2e}
\usepackage{hyperref}
\usepackage{titlesec}
\usepackage{siunitx}
\usepackage{caption}
\usepackage{longtable}
\usepackage[referable]{threeparttablex}
\allowdisplaybreaks

\newtheorem{theorem}{Theorem}[section]
\newtheorem{lemma}[theorem]{Lemma}
\newtheorem{definition}[theorem]{Definition}

\numberwithin{equation}{section}

\theoremstyle{remark}
\newtheorem{remark}[theorem]{Remark}

\titleformat{\section}
  {\normalfont\large\bfseries\centering}{\thesection}{1em}{}

\titleformat{\subsection}
  {\normalfont\bfseries}{\thesubsection}{1em}{}  
\titleformat{\subsubsection}
  {\normalfont\bfseries}{\thesubsubsection}{1em}{}

\def\reals{\hbox{\rm I\kern-.18em R}}
\def\complexes{\hbox{\rm C\kern-.43em
\vrule depth 0ex height 1.4ex width .05em\kern.41em}}
\def\field{\hbox{\rm I\kern-.18em F}} 

\newenvironment{section*}[2][A]{
  \section*{#2}
  \renewcommand\thesection{#1}
  \setcounter{theorem}{0}}{}

\begin{document}

\title[Divisor Problem]{Explicit and Effective Estimates for the error term in the Generalised Divisor Problem}

\author{Neea Paloj{\"a}rvi}
\address{School of Science, UNSW Canberra, Australia}
\email{n.palojarvi@unsw.edu.au}

\author{Sebastian Tudzi}
\address{School of Science, UNSW Canberra, Australia}
\email{s.tudzi@unsw.edu.au}
\date\today
\subjclass[2020]{Primary: 11N99}
\keywords{ Divisor function, analytic approach, explicit results.}
\begin{abstract}
    In this article, we obtain effective estimates for the error term $\Delta_{k}(x)$ for all integers $k \geq2$, and completely explicit estimates for integers $k \in [3,9]$. The explicit results improve the powers of $x$ appearing in the known explicit bounds for $\Delta_{k}(x)$, and the effective bounds provide a method to derive such bounds for all integers $k \geq 3$.
\end{abstract}

\maketitle

\section{Introduction}
Let the divisor function $d(n)$ denote the number of factors of a positive integer $n$, and let $d_{k}(n)$ denote the $k$-fold divisor function. In this paper, we aim to derive an explicit bound for the error term in the sum
\begin{equation*}
    T_{k}(x):=\sum_{n \le x} d_k(n)
\end{equation*}
for positive integers $k\ge 2$ using  analytic approach. In particular, we seek to write $T_{k}(x)$ in the form,
\begin{equation}\label{eq.tk}
    T_{k}(x)=xP_{k}(\log x)+\Delta_{k}(x),
\end{equation}
where $P_{k}(x)$ is a polynomial in $x$ of degree $k-1$ with leading coefficient $1/(k-1)!$ and $\Delta_{k}(x)$ is the error term. 

A typical upper bound estimate for $T_{k}(x)$ is given by 
\begin{equation}\label{nhr}
    T_{k}(x)\le\frac{x}{k-1}(\log x + h_{k})^{k-1},
\end{equation}
where $h_k$ is a function of $k$. Mardjanichvili \cite{mardjanichvili1939} and Nicolas and Tenenbaum (see \cite[p.~2]{bordelles2006}) proved that $h_{k}=k-1$ for $k\ge 1$ and $x\ge 1$. Bordell{\'e}s \cite{bordelles2006} improved this result to $h_{k}=k-2$ for $k\ge 3$ and $x\ge 13$. A further improvement was established by Dubbe \cite{Dubbe2020}. That is, for either $k\ge 3$ and $1\le x\le 3$, or $k\ge 2$ and $x\ge 3$, $h_{k}=k(3/2-\log 2)$. However, these results may not provide the sharpest possible estimates for $\Delta_k(x)$ as stated in \eqref{eq.tk}.

Dirichlet \cite{dirichlet1851bestimmung} used his famous hyperbola method to prove that for $k=2$,
\begin{equation}\label{fbmm}
T(x)=x(\log x + 2\gamma - 1) + \Delta(x)
\end{equation}
with $|\Delta(x)|\le \alpha x^{1/2}$ for all $x\ge x_{0}$, where $\alpha$ is  a positive constant and $\gamma$ is the Euler--Mascheroni constant. Subsequent works have focused on refining the exponent $x$ that appears in the error term $\Delta(x)$. Note that $T(x)=T_{2}(x)$ and $\Delta(x)=\Delta_{2}(x)$. Hardy \cite[p.81]{Hardy1999} conjectured  that the exponent of $x$ in $\Delta(x)$ is $1/4$. Nonetheless, Huxley \cite{MR2005876} provided the best-known bound to date, which yields an exponent of $131/416$. Even with these advancements, it is still quite difficult to obtain explicit constants connected to these results.  

Berkane, Bordell\`es, and Ramar\'e \cite[Corollary 2.1]{MR2869048} showed that obtaining explicit bounds is useful for applications in class number fields. They provided the following explicit pairs $(\alpha,\, x_{0})$ for $\Delta(x)$: $(0.961,\, 1)$, $(0.482,\, 1981)$, and $(0.397,\, 5560)$ (see \cite[Theorem 1.1]{MR2869048}). For $x\ge 9995$, they obtained $|\Delta(x)|\le 0.764x^{1/3}\log x$. Later, Simoni\v{c} and Starichkova \cite[footnote, p.~8]{MR4500746} verified that this finding holds for all $x\ge 5$. Using Dirichlet convolution, Bordell\`es \cite[Lemma 3.2]{MR1917797} and later Cully-Hugill and Trudgian \cite[Theorem 2]{MR4311680} obtained explicit bounds for the cases $k=3$ and  $k=4$ respectively, demonstrating that the hyperbola method and convolution identities yield the main term $xP_k(\log x)$ and an error term of size $\Delta_k(x)=O(x^{(k-1)/k+\varepsilon})$. Building on earlier results, Tudzi \cite{Tudzi2025} established and improved the existing explicit bounds that confirm this result for all $k\ge 3$.

While these convolution-based results provide precise explicit estimates, the focus of the present work is on the analytic approach, which allows leveraging properties of the Riemann zeta function and its generalizations to study \(\Delta_k(x)\) and potentially achieve sharper bounds for $x$ large enough. We follow the main steps of the proofs of \cite[Theorems 12.2 \& 12.3]{MR882550}. Although the analytic method is well known for deriving asymptotic results, we are not aware of any prior use of this method to derive effective or explicit estimates for the generalized divisor sum.

We describe our main explicit result below. In addition, in Section \ref{sec:Effective} we prove several effective results that give a tool to derive explicit estimates of sizes 
\begin{equation*}
    O_{k,\varepsilon}(x^{\frac{k-1}{k+2}+\varepsilon}) \text{ if } k\geq 4, \quad\text{and}\quad O_{k,\varepsilon}(x^{\frac{k+1}{k+4}+\varepsilon}) \text{ and } O_{k,\varepsilon} (x^{\frac{k-1}{k+1}+\varepsilon}) \text{ if } k\geq 2
\end{equation*} for the function $\Delta_k(x)$.

\begin{theorem}
\label{thm:main}
      For $3\leq k \leq 9$ and $1\leq j \leq 3$, we have
    \begin{equation*}
        \left|\Delta_k(x)\right|<\alpha_{k,j}(\log{x})^{\beta_{k,j}} x^{\gamma_{k,j}}
    \end{equation*}
    if $x_{k,j}\leq x<x_{k,j+1}$, where $j=1,2$, or $x \geq x_{k,3}$. Values $x_{k,j}$, $\alpha_{k,j}$, $\beta_{k,j}$ and $\gamma_{k,j}$ are presented in Table \ref{table:MainResults}. \ Note that if $k\notin \{5,6\}$, then we only have the case $j=3$ meaning that we have only one estimate in all of the cases.
    \begin{center}
\begin{tabular}{ |c|c|c|c| } 
 \hline
 $k$ & $j$ & $x_{k,j}$ & $(\alpha_{k,j},\beta_{k,j},\gamma_{k,j})$ \\ \hline
 $3$ & $3$& $1.601\cdot10^{98}$ & $(0.248,3, 859/1400)$\\ \hline
 $4$ & $3$& $2.855\cdot10^{41}$ & $(0.084,5, 3/5)$\\ \hline
  & $1$  & $1667$ & $(9.272,41/7, 11/14)$\\
 $5$ & $2$  & $1.756\cdot10^7$ & $(0.645,41/7, 24/35)$ \\ 
  & $3$  & $1.601\cdot10^{29}$ & $(0.021, 41/7, 556/875)$ \\ \hline
    & $1$  & $\num{162727}$ & $(0.274,  27/4, 33/40)$\\
 $6$ & $2$  & $1.601\cdot10^{39}$ & $(0.006,27/4, 2929/4000)$ \\ 
  & $3$  & $1.150\cdot10^{72}$ & $(0.004,  27/4, 35/48)$ \\ \hline
 $7$ & $3$& $3.129\cdot10^{20}$ & $(0.004, 23/3, 5/6)$\\ \hline
 $8$ &  $3$& $3.698\cdot10^{65}$ & $(0.001,43/5, 43/50)$\\ \hline
 $9$ &  $3$& $1.601\cdot10^{72}$ & $(2.100\cdot10^{-5}, 105/11, 1949/2200)$\\ \hline
\end{tabular}
\captionof{table}{Values $x_{k,j}$, $\alpha_{k,j}$, $\beta_{k,j}$ and $\gamma_{k,j}$}\label{table:MainResults}
\end{center}
\end{theorem}

As we can see from Table \ref{table:MainResults}, the lower bound for $x$ does not grow together with $k$. This is because of two main reasons. Firstly, we apply different methods to derive the results (different moments of the Riemann zeta function and integral estimates) and choose the best estimate for each range of $x$ and each value of $k$. Secondly, we have only included those values that would improve the best known explicit estimates. For example, in Section \ref{sec:ProofMain}, we derive estimates for $|\Delta_3(x)|$ for all $x \geq \num{176749}$. However, those results improve the existing estimates only if $x \geq 1.601\cdot10^{98}$, and hence only that result is included in Theorem \ref{thm:main}.

 To illustrate the strength of the results in Theorem \ref{thm:main}, we compare them to other explicit results for the generalized divisor problem. These are presented in Table \ref{table:comparision}. Note that result \cite[Theorem 1.2]{Tudzi2025} assumes that $x\geq 10^{12}$ if $k=5$ and $x \geq 10^{13}$ if $k=6$. Hence, Theorem \ref{thm:main} extends the region of $x$ where an estimate for $|\Delta_k(x)|$ is provided for those values of $x$. Note that we have only compared our results to completely explicit estimates. Hence, there is no comparison with \cite[Theorem 1.2]{Tudzi2025} for $k \geq 7$. However, since the exponent of $x$ in Theorem \ref{thm:main} is smaller than the exponent appearing in \cite[Theorem 1.2]{Tudzi2025}, Theorem \ref{thm:main} will improve \cite[Theorem 1.2]{Tudzi2025} for all $x$ large enough.

    \begin{center}
\begin{tabular}{ |c|c|c| } 
 \hline
 $k$ & Range of $x$ & Result \\ \hline
 $3$ & $2\leq x <1.601\cdot10^{98}$ & Tudzi \cite[Theorem 1.1]{Tudzi2025} \\ 
 $3$ & $x\geq 1.601\cdot10^{98} $ & Theorem \ref{thm:main} \\ \hline
  $4$ & $2\leq x <2.855\cdot10^{41}$ &  Cully-Hugill \& Trudgian \cite[Theorem 1]{MR4311680} \\ 
 $4$ & $x\geq 2.855\cdot10^{41}$ & Theorem \ref{thm:main} \\ \hline
 $5$ & $1667 \leq x <10^{12}$ & Theorem \ref{thm:main} \\ 
 $5$ & $10^{12} \leq x < 1.601\cdot10^{29}$ & Tudzi \cite[Theorem 1.2]{Tudzi2025} \\ 
 $5$ & $x\geq 1.601\cdot10^{29}$ & Theorem \ref{thm:main} \\ \hline
  $6$ & $\num{162727} \leq x <10^{13}$ & Theorem \ref{thm:main} \\ 
 $6$ & $10^{13} \leq x < 1.601\cdot10^{39}$ & Tudzi \cite[Theorem 1.2]{Tudzi2025} \\ 
 $6$ & $x\geq 1.601\cdot10^{39}$ & Theorem \ref{thm:main} \\ \hline
 $7$ & $10^{14}\leq x <3.129\cdot10^{20}$ & Tudzi \cite[Corollary 1.3]{Tudzi2025} \\ 
 $7$ & $x\geq 3.129\cdot10^{20}$ & Theorem \ref{thm:main} \\ \hline
 $8$ & $10^{14}\leq x <3.698\cdot10^{65}$ & Tudzi \cite[Corollary 1.3]{Tudzi2025} \\ 
 $8$ & $x\geq 3.698\cdot10^{65}$ & Theorem \ref{thm:main} \\ \hline
 $9$ & $10^{14}\leq x <1.601\cdot10^{72}$ & Tudzi \cite[Corollary 1.3]{Tudzi2025} \\ 
 $9$ & $x\geq 1.601\cdot10^{72}$ & Theorem \ref{thm:main} \\ \hline
\end{tabular}
\captionof{table}{The best known explicit results for $|\Delta_k(x)|$ if $k=3,4,\ldots,9$.}\label{table:comparision}
\end{center}

The outline of the rest of the paper is as follows. In Section, \ref{sec:TkIntegralSum}, we show that estimating the function $T_{k}(x)$ reduces to bounding an integral involving the Riemann zeta function together with a sum depending on $d_{k}(x)$; we also obtain an estimate for this sum. In Section \ref{sec:Integrals}, we analyse the integral in detail, identifying the main term $P_{k}(x)$ and bounding the remaining contributions using three different approaches: the fourth moment and the second moment of $\zeta(1/2+it)$  as well as an integral estimate derived from the functional equation of the zeta function. In Section \ref{sec:Effective}, we establish effective versions of our results, and in Section \ref{sec:ProofMain}, we prove Theorem \ref{thm:main}. Finally, in Section \ref{Disc}, we conclude with a brief discussion.
\section{Case distinction: Integer and non-integer $x$}
\label{sec:TkIntegralSum}
In this section, we establish an estimate that will be useful in analysing the behaviour of $T_k(x)$. The result relies on distinguishing between integer and non-integer values of $x$.
\begin{lemma}
\label{lemma:TkIntegral}
Assume $T>c>0$. If $x$ is not an integer, then we have
\begin{equation}
\label{eq:TkSum}
    \left|T_k(x)-\frac{1}{2\pi i}\int_{c-iT}^{c+iT} \zeta^k(w)\frac{x^w}{w} \, \mathrm{d}w\right| <\frac{1+\pi}{\pi}\cdot\frac{x^c}{T}\sum_{n=1}^\infty \frac{d_k(n)}{n^c \left|\log{\frac{x}{n}}\right|}. 
\end{equation}
Moreover, if $x$ is an integer, then we have
\begin{multline}
\label{eq:integerdSumFirst}
    \left|T_k(x)-\frac{1}{2\pi i}\int_{c-iT}^{c+iT} \zeta^k(w)\frac{x^w}{w} \, \mathrm{d}w\right| \\
    <\frac{1+\pi}{\pi}\cdot\frac{x^c}{T}\sum_{\substack{n=1 \\ n \neq x}}^\infty \frac{d_k(n)}{n^c \left|\log{\frac{x}{n}}\right|}+\frac{d_k(x)}{2\pi}\left(\pi+\frac{2Tc}{T^2-c^2}\right). 
\end{multline}
\end{lemma}

\begin{proof}
First, we write the integral on the left-hand side of the claim by using a series expansion and obtain
\begin{equation*}
    \frac{1}{2\pi i}\int_{c-iT}^{c+iT} \zeta^k(w)\frac{x^w}{w} \, \text{d}w=\sum_{n=1}^\infty \frac{d_k(n)}{2\pi i} \int_{c-iT}^{c+iT} \frac{(x/n)^w}{w} \, \text{d}w.
\end{equation*}
Using \cite[Theorem 15]{Estermann1952} and \cite[Lemma 2.1]{Dudek2014} we obtain the desired bound if $x$ is not in integer.
Moreover, if $x$ is an integer, we have to add the case $x=n$, too. In this case, we simply write that
\begin{equation*}
    \left|\int_{c-iT}^{c+iT} \frac{(x/n)^w}{w} \,\right| \mathrm{d}w=\left|\log{\left(\frac{c+iT}{c-iT}\right)}\right|\leq \pi+\frac{2Tc}{T^2-c^2},
\end{equation*}
and add the corresponding term to the estimate.
\end{proof}

Next, we consider the last sum in \eqref{eq:TkSum}. In order to prove the desired size of the main results, we would like the sum to be of size $O(x^{1+\varepsilon-c})$ for any $\varepsilon>0$. Hence, we need good enough explicit bounds for the function $\log{d_k(n)}$. By classical results, we know that $\log{d_k(n)}=O(\log{n}/\log\log{n})$ (see \cite{NORTON1992,Wigert1907}). Thus, we give the following definition for $\lambda(k)$.

\begin{definition}
\label{def:lambdak}
   Let $k\geq 2$ be fixed. We define $\lambda(k) \in \mathbb{R}_+$ be a real number such that 
   \begin{equation}
   \label{eq:dkUpper}
       \log{d_k(n)} \leq \frac{\lambda(k)\log{n}}{\log\log{n}}
   \end{equation}
   for all $n \geq 3$.
\end{definition}

\begin{remark}
    In the pre-print \cite{Teo2025}, Teo provides values for $\lambda(k)$ when $k \in [2,100]$ and an efficient algorithm to compute values for other $k$, too.
\end{remark}

\begin{remark}
    Note that $d_k(3)=k$ for all $k \geq 2$ and hence in \eqref{eq:dkUpper} we must have $\lambda(k) \geq k(\log\log{3})(\log{3})/\log{3}>\log\log{3}$ for all $k \geq 3$. This is essential in Lemma \ref{lemma:dkGeneral} where we need to ensure that there exists $\varepsilon$ in $(0, \lambda(k)/\log\log{3}-1)$. 
\end{remark}

The lemmas presented here addresses the estimation of the sum appearing on the right-hand side of \eqref{eq:TkSum}. Before proving an estimate for this sum, we define functions $r_i$, $i=1,\ldots, 4$.

Let us denote 
\begin{align}
    &r_1(x,k,\varepsilon, c, a):= \frac{(1+c)(x/a)^{1-c}(\log{(x/a)}+h_k)^{k-1}}{\log{a}} \label{def:r1}\\
    &\quad+\frac{1}{\log{a}}\left(\frac{1}{2(ax)^{c-\varepsilon}}+\frac{(ax)^{1+\varepsilon-c}}{c-\varepsilon-1}+\frac{c-\varepsilon}{12(ax)^{1+c-\varepsilon}}+\frac{(c-\varepsilon)(1+c-\varepsilon)(2+c-\varepsilon)}{720(ax)^{c-\varepsilon+3}}\right), \nonumber \\
    & r_2(x,a):=\log{\left(x\left(1-\frac{1}{a}\right)-\frac{1}{2}\right)}+\gamma+\frac{1}{2x(1-1/a)-1}, \label{def:r2} \\
    & r_3(x,a):=\log{x}+\log{(a-1)}+\gamma+\frac{1}{2(a-1)x} \quad\text{and}\label{def:r4}\\
    & r_4(x,a):=\log{x}+\log{\left(1-\frac{1}{a}\right)}+\gamma-\frac{1}{2}+\frac{1}{2a}+\frac{1}{2x(1-1/a)}. \label{def:r6}  
\end{align}
Here, $h_k$ is as in \eqref{nhr}. If we have $\int_a^b f(x) \, \text{d}x$, where $b<a$, we interpret it to be $-\int_b^a f(x) \, \text{d}x$. Note that now $r_i \ll O(x^{\varepsilon'})$ as $x \to \infty$ for any $\varepsilon'>0$. Hence, Lemma \ref{lemma:dkGeneral} provides upper bounds of size $O(x^{1+\varepsilon'-c})$ for the last sum in \eqref{eq:TkSum} for any $\varepsilon'>0$, as desired.
\begin{lemma}
\label{lemma:dkGeneral}
Let $k \geq 2$,  $h_k:=k-1$ if $k=2$ and $k-2$ otherwise, $\varepsilon \in (0, \lambda(k)/\log\log{3}-1)$, $c$ be any real number in $( 1+\varepsilon,\lambda(k)/\log\log{3})$ and $x_0 \geq 3$ be such that $\lambda(k)/\log\log{x_0} \leq \varepsilon$. Assume $x \geq ax_0$, $a>1$, $x-1/2 \geq e^e$, $x>3a$ and $x(1-1/a)\geq 3/2$. 

If $x$ is not an integer and $N$ is the closest integer to $x$ (we choose the smaller one if there are two equally far apart), then 
\begin{align}
    &\sum_{n=1}^\infty \frac{d_k(n)}{n^c \left|\log{\frac{x}{n}}\right|} < r_1(x,k,\varepsilon, c, a)+N^{\varepsilon-c}\frac{2N+|x-N|}{|x-N|} \label{eq:sumdKNonInteger}  \\
     &\quad+\left(2x-\frac{1}{2}\right)\left(\frac{x}{a}\right)^{\varepsilon-c}r_2(x,a) +(a+1)x\left(x+\frac{1}{2}\right)^{\varepsilon-c}r_3(x,a).\nonumber 
\end{align}

If $x$ is an integer, then
\begin{align}
    &\sum_{\substack{n=1 \\ n \neq x}}^\infty \frac{d_k(n)}{n^c \left|\log{\frac{x}{n}}\right|} < r_1(x,k,\varepsilon, c, a)\label{eq:sumDkInteger} \\
    &\quad+\left(x-\frac{1}{2}\right)\left(\frac{x}{a}\right)^{\varepsilon-c}r_4(x,a)+\left(x+\frac{1}{2}\right)(x+1)^{\varepsilon-c}\left(r_3(x,a)+\frac{1}{2}\right).  \nonumber
\end{align}
\end{lemma}

\begin{proof}
    We will divide the sum into different cases based on how close $x$ and $n$ are together. If $n<x/a$ or $n>ax$, then $|\log{(x/n)}|>\log{a}$. Hence and using the upper bound for $d_k(n)$ provided in \eqref{eq:dkUpper}, we have
    \begin{align}
      \sum_{\substack{n<x/a, \\ n \geq ax}} \frac{d_k(n)}{n^c\left|\log{\frac{x}{n}}\right|}&\leq \frac{1}{\log{a}}\sum_{\substack{n<x/a, \\ n \geq ax}} \frac{d_k(n)}{n^c}. 
       \label{eq:nLargeSmall}
    \end{align}
   
    By estimate \eqref{nhr} and partial summation, the first sum on the right-hand side of \eqref{eq:nLargeSmall} can be bounded by
    \begin{align*}
         & = \frac{1}{\log{a}}\left(\frac{T_k(x/a)}{(x/a)^c}+c\int_1^{a/x}\frac{T_k(t)}{t^{c+1}} \, dt\right) \leq \frac{(1+c)(x/a)^{1-c}(\log{(x/a)}+h_k)^{k-1}}{\log{a}}
    \end{align*}
    By Euler--Maclaurin formula, Definition \ref{def:lambdak} and since $c \in ( 1+\varepsilon,\lambda(k)/\log\log{4})$, 
    and the last sum on the right-hand side of \eqref{eq:nLargeSmall} can be bounded by
     \begin{align*}
         \leq \frac{1}{\log{a}}\left(\frac{1}{2(ax)^{c-\varepsilon}}+\frac{(ax)^{1+\varepsilon-c}}{c-\varepsilon-1}+\frac{c-\varepsilon}{12(ax)^{1+c-\varepsilon}}+\frac{(c-\varepsilon)(1+c-\varepsilon)(2+c-\varepsilon)}{720(ax)^{c-\varepsilon+3}}\right). 
    \end{align*}
    Hence, the sum \eqref{eq:nLargeSmall} can be estimated as
    \begin{equation*}
         \sum_{\substack{n<x/a, \\ n \geq ax}} \frac{d_k(n)}{n^c\left|\log{\frac{x}{n}}\right|} \leq  r_1(x,k,\varepsilon, c, a).
    \end{equation*}

    Let us now consider the case $x/a\leq n \leq ax$. First, we will look at the case where $x$ is not an integer and $ax\geq n>N$. Let us write $n=N+r$. Note that in this case $r \geq 1$, $ax\geq n=N+r>x$ and $N \geq x-1/2$. Since $\log{(y+1)} \geq \frac{2y}{2+y}$ for all $y \geq 0$, we obtain
    \begin{equation*}
        \left|\log{\frac{x}{n}}\right|=\log{\left(\frac{N+r}{x}\right)} \geq \frac{2(N+r-x)}{N+r+x} \geq \frac{2(r-1/2)}{(a+1)x}.
    \end{equation*}
   Since $N+1\geq x+1/2\geq ax_0+1/2>x_0$, we have
    \begin{equation}
         \sum_{N+1\leq n \leq ax} \frac{d_k(n)}{n^c \left|\log{\frac{x}{n}}\right|} < (a+1)x\sum_{N+1\leq n \leq ax} \frac{n^{\lambda(k)/\log\log{x_0}-c}}{2(n-N-1/2)} \label{eq:1r}.
    \end{equation} 
    The sum on the right-hand side of \eqref{eq:1r} can be written as 
      $$
     \sum_{N+1\leq n \leq ax}\frac{1}{2(n-N-1/2)}=\sum_{1\leq r \leq (a-1)x} \frac{1}{2(r-1/2)}.
    $$ 
    However, we rewrite the sum as a harmonic sum, which allows us to apply known explicit estimates. Thus, we estimate
    $$
    \frac{1}{2(r-1/2)}=\frac{r}{2(r-1/2)}\cdot \frac{1}{r}=\frac{1}{2(1-1/(2r))}\cdot \frac{1}{r} \leq \frac{1}{r}.
    $$
    Thus, using using \cite{Young1991} (see also \cite[p.~75]{Havil2003}) gives 
    \begin{align*}
         &\sum_{N+1\leq n \leq ax} \frac{d_k(n)}{n^c \left|\log{\frac{x}{n}}\right|} \leq (a+1)x\left(x+\frac{1}{2}\right)^{\varepsilon-c}\sum_{1 \leq r \leq (a-1)x} \frac{1}{r} \\
        &\quad \leq (a+1)x\left(x+\frac{1}{2}\right)^{\varepsilon-c}\left(\log{x}+\log{(a-1)}+\gamma+\frac{1}{2(a-1)x}\right). 
    \end{align*}
    
    We can apply the same argument for the case $x/a \leq n<N \leq x+1/2$ and $x$ is not an integer using $n=N-r$ and 
    \begin{equation*}
        \left|\log{\frac{x}{n}}\right| \geq \frac{2(r-1/2)}{2x-1/2}.
    \end{equation*}
    Hence, since $1 \leq r \leq (1-1/a)x-1/2$, we have
    \begin{multline*}
        \sum_{x/a\leq n<N } \frac{d_k(n)}{n^c \left|\log{\frac{x}{n}}\right|} \\
        \leq \left(2x-\frac{1}{2}\right)\left(\frac{x}{a}\right)^{\varepsilon-c}\left(\log{\left(x\left(1-\frac{1}{a}\right)-1/2\right)}+\gamma+\frac{1}{2x(1-1/a)-1}\right).
    \end{multline*}
    
    If $n=N$, keeping in mind that $N \geq x-1/2 \geq e^e$ and hence $\log{y}/\log\log{y}$ is increasing for all $y$, we have
    \begin{align*}
        \frac{d_k(n)}{n^c \left|\log{\frac{x}{n}}\right|} &\leq N^{\varepsilon-c}\frac{2N+|x-N|}{|x-N|}. 
    \end{align*}

    Lastly, we need to consider the case where $x$ is an integer and $x/a<n\leq ax$. Since in this sum we have $n \neq x$, we can consider cases $ax\geq n \geq x+1$ and $x/a\leq n \leq x-1$. Now, we have
    \begin{equation*}
        \left|\log{\frac{x}{n}}\right| \geq \begin{cases}
        \frac{2r}{2x-r}, &\text{if } x>n \\
        \frac{2r}{2x+r}, &\text{if } x<n,
        \end{cases} 
    \end{equation*}
    where $n=x+r$ or $n=x-r$ and $r \geq 1$.
    Similarly to the case where $x$ is not an integer, we have
    \begin{align*}
        &\sum_{\substack{x/a\leq n \leq x-1 \\ x+1 \leq n \leq ax}} \frac{d_k(n)}{n^c \left|\log{\frac{x}{n}}\right|} \\
        &\quad < \left(\frac{x}{a}\right)^{\varepsilon-c}\left(x-\frac{1}{2}\right) \left(\left(\log{\left(x\left(1-\frac{1}{a}\right)\right)}+\gamma+\frac{1}{2x(1-1/a)}\right)-\frac{1}{2}+\frac{1}{2a}\right) \\
        &\quad\quad+(x+1)^{\varepsilon-c}\left(x+\frac{1}{2}\right) \left(\log{x}+\log{(a-1)}+\gamma+\frac{1}{2}+\frac{1}{2(a-1)x}\right)
    \end{align*}
    by \cite{Young1991}. The desired results follow by combining the above estimates.  
\end{proof}

\begin{remark}
    We compute explicit values for quite small values of $k$, indeed $k<10$. In these cases, Bordell{\'e}s' \cite{bordelles2006} estimate for $h_k$ is better than Dubbe's \cite{Dubbe2020}. For $k \geq 11$ one may want to replace $h_k$ in Lemma \ref{lemma:dkGeneral} with Dubbe's estimate $h_k \leq k(3/2-\log{2})$.
\end{remark}

\section{Decomposition and estimation of  integrals}
\label{sec:Integrals}
In Lemma~\ref{lemma:TkIntegral}, we showed that $T_{k}(x)$ can be estimated by an integral. The remaining task is therefore to bound the size of this integral. The next lemma identifies the terms that require estimation, after which we estimate each of the resulting integrals.

\begin{lemma}
\label{lemma:IntegralsToBeEstimated}
Let $c>1$, $b<1$ and $T>0$. If $b>0$, then we have
\begin{multline}
\label{eq:zetakRectangle}
\frac{1}{2\pi i}\int_{c-iT}^{c+iT} \zeta^k(w)\frac{x^w}{w} \, \mathrm{d}w \\
=xP_k(\log{x})-\frac{1}{2\pi i}\left(\int_{c+iT}^{b+iT}+\int_{b+iT}^{b-iT}+\int_{b-iT}^{c-iT}\right) \zeta^k(w)\frac{x^w}{w} \, \mathrm{d}w.
\end{multline}
In the case $b<0$, we add $(-2)^{-k}$ to the right-hand side of the equation \eqref{eq:zetakRectangle}.
\end{lemma}

\begin{proof}
Let us consider a rectangle with vertices $-b-iT$, $c-iT$, $c+iT$ and $b+iT$. The residue at $w=1$ is $xP_k(\log{x})$, and if $b<0$, there is a residue $\zeta(0)^k=(-2)^{-k}$ at $w=0$. The claim follows.
\end{proof}

Next, we estimate the integrals in Lemma \ref{lemma:IntegralsToBeEstimated} in the next three subsections. 

\subsection{Integrals in intervals $[b,c]$}
 We note that we require estimates for $|\zeta(s)|$ for $b\leq\Re(s)\leq c$ and $|\Im(s)|=T$. In the next three lemmas, we provide these estimates. We  split the interval $[b,c]$ into two sub-intervals: $[-b,1/2]$ and $[1/2,c]$. The first lemma considers the case $[1/2,c]$.

\begin{lemma}
\label{lemma:zetaMiddle}
Assume $c>1$, $s:=\sigma+it$, $t\geq t_0 \geq 3$ and $\sigma \in [1/2,c]$. Then
\begin{equation*}
     \left|\zeta(s)\right|\leq a_1(c,\sigma,t_0)t^{\frac{c-\sigma}{6(c-0.5)}}\log{t},
\end{equation*}
where 
\begin{multline}
\label{def:a1}
    a_1(c,\sigma, t_0):=0.611^{\frac{c-\sigma}{c-0.5}}\left(\frac{\zeta\left(c\right)}{\log{t_0}}\right)^{\frac{\sigma-0.5}{c-0.5}}\left(1+\frac{\pi}{2\log{t_0}}+\frac{\pi(c+1.31)^2}{4t_0(\log{t_0})^2}+\frac{c+1.31}{2t_0^2\log{t_0}}\right) \\
    \cdot \left(\frac{c+1.31+t_0}{t_0}\right)^{\frac{c-\sigma}{6(c-0.5)}}. 
\end{multline}
\end{lemma}

\begin{proof}
We apply \cite[Lemma 4]{Leong2024}. As in the proof of Corollary 7 in \cite{Leong2024} but applying \cite[Theorem 1.1(a)]{R2026} instead of \cite[Theorem 1.1]{HPY2022}, we observe that choosing $(k_1,k_2):=(0.611,1/6)$ and $Q_0=1.31$ is appropriate if $\sigma=1/2$ in \cite[Lemma 4]{Leong2024}. Even more, we have
\begin{equation*}
    \left|c-1+it\right|\left|\zeta(1.5+it)\right| \leq \frac{\zeta(c)}{\log{t}}\left|c-1+it\right|\log{t},
\end{equation*}
and hence we can choose $(k_3,k_4,Q_0):=(1,\zeta(c)/\log{t_0},1.31)$ in \cite[Lemma 4]{Leong2024} since $t \geq t_0$. To simplify our results in this Lemma and later in Lemma \ref{lemma:sigmaChanges}, we choose $\sigma=c$ in $a_0(\sigma, Q_0, t)$ and $a_1(\sigma, Q_0, t)$ in \cite[Lemma 4]{Leong2024}.
\end{proof}

\begin{remark}
\label{rmk:a3}
    If we assume that $t_0\geq 8.97\cdot10^{17}$ in Lemma \ref{lemma:zetaMiddle}, then we can replace $0.611$ in $a_1$ by $0.566$ using \cite[Theorem 1.1(b)]{R2026}. We will denote this by $a_3(c,\sigma, t_0)$. 
\end{remark}

Now we provide a similar estimate in the case $[b,1/2]$. 

\begin{lemma}\cite[Theorem 4]{Rademacher1959}, \cite[Lemma 3]{Trudgian2015}
\label{lemma:Rademacher}
    Let $-1/2\leq -\eta \leq \sigma \leq 1+\eta\leq 3/2$. Then we have
    \begin{equation*}
        \left|\zeta(s)\right| \leq  \frac{1+\eta}{1-\eta}\left|\frac{1+s}{1-s}\right|\left(\frac{|1+s|}{2\pi}\right)^{\frac{1-\sigma+\eta}{2}}\zeta(1+\eta).
    \end{equation*}
\end{lemma}

The previous lemma gives a general estimate for the Riemann zeta function in a large interval. However, we would like to have an estimate for $\sigma \in [b,1/2]$ that is of size $O(t^{1/6}\log{t})$ when $\sigma=1/2$. As before, this can be derived by combining known estimates from the lines $\sigma=-1/12$ and $\sigma=1/2$ using the Phragm\'en--Lindel\"of principle.

\begin{lemma}
\label{lemma:smallb}
    Assume $-0.31<b<0$, $s:=\sigma+it$, $t\geq t_0 \geq 3$ and $\sigma \in [b,1/2]$. Then
\begin{equation*}
    \left|\zeta(s)\right|\leq a_2(b,\sigma, t_0)t^{\frac{4(3b\sigma-\sigma-2b)+3}{6(1-2b)}}\log{t}, 
\end{equation*}
where
\begin{multline*}
    a_2(b,\sigma, t_0):=0.611^{\frac{\sigma-b}{0.5-b}}\left(\frac{(1-b)\zeta\left(1-b\right)}{(1+b)(2\pi)^{\frac{1-2b}{2}}\log{t_0}}\right)^{\frac{0.5-\sigma}{0.5-b}}\left(\frac{1.81+t_0}{3}\right)^{\frac{4(3b\sigma-\sigma-5b)+9}{6(1-2b)}} \\
    \cdot \frac{\log{\left((1.81+t_0)\cdot t_0^{-1}\right)}}{\log{t_0}}.
\end{multline*}
\end{lemma}

\begin{proof}
    The proof is similar to the proof of Lemma \ref{lemma:zetaMiddle} but now we also apply Lemma \ref{lemma:Rademacher}, and Lemma 3 instead of Lemma 4 in \cite{Leong2024} (see also \cite{Fiori2025,Trudgian2014}). We choose $f(s):=(s-1)\zeta(s)$, in \cite[Lemma 3]{Leong2024}. By Lemma \ref{lemma:Rademacher}, choosing $\eta=-b$ for $t \geq t_0$ gives
    \begin{align*}
        &\left|f(b+it)\right|\leq  \frac{1-b}{1+b}\left|\frac{1+b+it}{1-b+it}\right|\left(\frac{|1-b+it|}{2\pi}\right)^{\frac{1-2b}{2}}\zeta\left(1-b\right)\left|1-b+it\right| \\
        &\quad<\frac{1-b}{(1+b)(2\pi)^{\frac{1-2b}{2}}}\zeta\left(1-b\right)\left|1-b+it\right|^{\frac{3-2b}{2}} \\
        &\quad \leq \frac{1-b}{(1+b)}(2\pi)^{\frac{1-2b}{2}}\log{t_0}\zeta\left(1-b\right)\left|1-b+it\right|^{\frac{3-2b}{2}}\log{\left|1.31+b+it\right|}.
    \end{align*}
Thus, we can choose
\begin{equation*}
    (A,B, \alpha_1, \alpha_2, \beta_1, Q_0):=\left(\frac{(1-b)\zeta\left(1-b\right)}{(1+b)(2\pi)^{\frac{1-2b}{2}}\log{t_0}}, 0.611,\frac{3-2b}{2}, 1, \frac{7}{6}, 1.31\right)
\end{equation*}
in \cite[Lemma 3]{Leong2024}. In order to derive final results, we simplify
\begin{equation*}
    1<\left|\frac{Q+\sigma}{t}+i\right|\leq \frac{1.81+t_0}{t_0} \text{ and } \frac{|\log{\left(t^{-1}\left|Q+s\right|\right)}|}{\log{t}} \leq \frac{\log{\left((1.81+t_0)\cdot t_0^{-1}\right)}}{\log{t_0}}.
\end{equation*}
The claim follows by \cite[Lemma 3]{Leong2024}.
\end{proof}

Now we have derived all essential lemmas to estimate the first and the last integral in Lemma \ref{eq:zetakRectangle}. In order to describe the results, let us define
\begin{align}
    & r_5(x,b,c,T,T_0, T_1, k) \notag\\
    &\quad:=\frac{(\log{T_1})^k}{T\pi}\left(\frac{1}{2}-b\right)\max\left\{a_2(b,b,T_0)^kT^{\frac{(1-2b)k}{2}}x^b, a_2\left(b, \frac{1}{2},T_0\right)^kT^{\frac{k}{6}}x^{\frac{1}{2}}\right\} \label{def:r7}\\
    &\quad\quad+ \frac{(\log{T_1})^k}{T\pi}\left(c-\frac{1}{2}\right)\max\left\{a_1\left(c,\tfrac{1}{2},T_0\right)^kT^{\frac{k}{6}}x^{\frac{1}{2}},a_1\left(c,c,T_0\right)^k x^c\right\}. \notag
\end{align}
The next lemma provides estimates for the first and the last integral in Lemma \ref{eq:zetakRectangle}.

\begin{lemma}
\label{lemma:sigmaChanges}
Assume $k \geq 0$, $T_1 \geq T \geq T_0\geq  3$, $c >1$ and $b\in [-0.31,0)$ or $b=1/2$. Then
    \begin{equation*}
      \left|\frac{1}{2 \pi i}\left(\int_{c+iT}^{b+iT}+\int_{b-iT}^{c-iT}\right) \zeta^k(w)\frac{x^w}{w} \, \mathrm{d}w\right| \leq  r_5(x,b,c,T,T_0,T_1,k).
    \end{equation*}
\end{lemma}

\begin{proof}
    Since the Riemann zeta function is symmetric with respect to the real line, we have
    \begin{multline}
    \label{eq:IntegralsHorizontal}
\left|\frac{1}{2 \pi i}\left(\int_{b+iT}^{c+iT}+\int_{b-iT}^{c-iT}\right)  \zeta^{k}(s) \frac{x^s}{s} \text{d}s\right| \leq \frac{1}{\pi}\int_{b}^c \left|\zeta(\sigma+iT)\right|^k \frac{x^{\sigma}}{T} \, \text{d}\sigma  \\
= \frac{1}{\pi}\left(\int_{b}^{1/2}+\int_{1/2}^{c}\right) \frac{\left|\zeta(\sigma+iT)\right|^k x^\sigma }{T} \, \text{d}\sigma.
\end{multline}
We will first look at the last integrals on the right-hand side of \eqref{eq:IntegralsHorizontal}, since they are the same in both cases.

We begin with the second integral. By Lemma \ref{lemma:zetaMiddle}, we have
\begin{equation*}
    \frac{\left|\zeta(\sigma+iT)\right|^k x^\sigma }{T}  \leq \exp\left(\sigma\log{x}+k\log{a_1(c,\sigma,T_0)}+\left(\tfrac{k(c-\sigma)}{6(c-0.5)}-1\right)\log{T}+k\log\log{T}\right).
\end{equation*}
Note that we can write
$
    \log{a_1(c,\sigma,T_0)}=(c-\sigma)\log{A}+(\sigma-0.5)\log{B}+\log{C},
$
where the terms $A,B$ and $C$ do not depend on $\sigma$. Hence, the previous exponent is a linear function of $\sigma$, and the integrand attains its maximum value either at $\sigma=1/2$ or $\sigma=c$. Thus, we have
\begin{multline}
\label{eq:secondIntbc}
   \frac{1}{\pi}\int_{1/2}^c \frac{\left|\zeta(\sigma+iT)\right|^k x^\sigma }{T} \, \text{d}\sigma \\
   \leq \frac{(\log{T_1})^k}{T\pi}\left(c-\frac{1}{2}\right)\max\left\{a_1\left(c,\tfrac{1}{2},T_0\right)^kT^{\frac{k}{6}}x^{\frac{1}{2}},a_1\left(c,c,T_0\right)^k x^c\right\}
\end{multline}
for $T \geq T_0$.

Now we consider the first integral on the right-hand side of \eqref{eq:IntegralsHorizontal}. If $b=1/2$, it is zero. For $b<0$, we apply Lemma \ref{lemma:smallb}. Similarly as in the previous case, we can conclude that the maximum of the integrand is attained at one of the endpoints. Hence, for $T\geq T_0$ we have
\begin{multline}
\label{eq:FirstIntbc}
    \frac{1}{\pi}\int_{b}^{1/2} \frac{\left|\zeta(\sigma+iT)\right|^k x^\sigma }{T} \, \text{d}\sigma \\
    \leq \frac{(\log{T_1})^k}{T\pi}\left(\frac{1}{2}-b\right)\max\left\{a_2(b,b, T_0)^kT^{\frac{(1-2b)k}{2}}x^b, a_2\left(b,\tfrac{1}{2}, T_0\right)^kT^{\frac{k}{6}}x^{\frac{1}{2}}\right\},
\end{multline}
and the claim follows by combining estimates \eqref{eq:IntegralsHorizontal}, \eqref{eq:secondIntbc} and \eqref{eq:FirstIntbc}.
\end{proof}

\begin{remark}
\label{rmk:newa6}
    If we assume that $T_0 \geq 8.97\cdot10^{17}$, $b=1/2$ and $c=1+\varepsilon_1$ in Lemma \ref{lemma:sigmaChanges}, then by Remark \ref{rmk:a3} we can replace $r_5$ with
\begin{align}
    & r_6(x,T,T_0, T_1, k,1+\varepsilon_1) \notag\\
    &:= \frac{(\log{T_1})^k}{T\pi}\left(\tfrac{1}{2}+\varepsilon_1\right)\max\left\{a_3\left(1+\varepsilon_1,\tfrac{1}{2},T_0\right)^kT^{\frac{k}{6}}x^{\frac{1}{2}},a_3\left(1+\varepsilon_1,1+\varepsilon_1,T_0\right)^k x^{1+\varepsilon_1}\right\}. \label{def:newr6}
\end{align}
\end{remark}

\subsection{Applying moment and pointwise bounds on $\zeta(1/2+it)$ }
\label{sec:momentEstimates}
In the lemmas that follow, we make use of explicit moments and pointwise estimates of $\zeta(1/2+it)$ to derive estimates for certain integrals that arise in our analysis. To begin, we estimate the integral using existing explicit bounds on the fourth moment of $\zeta(1/2+it)$. Here, we consider different ranges of $T$.
\begin{lemma}
\label{lemma:fourthmomentEst}
    Let $x>1$, $\sigma=1/2$, and $T \geq 5.5\cdot 10^7$. For any positive integer $k\ge 4$ we have 
    \begin{equation*}
|I(x,T)|:=\left|\frac{1}{2\pi i}\int_{\sigma - iT}^{\sigma + iT} \zeta^{k}(s) \frac{x^s}{s} \mathrm{d}s\right|< x^{1/2}[H_{k}(T)+c_{k}],
\end{equation*}
where the constant
\begin{align}\label{ckconstant}
c_{k}&:=1.039(1.461)^{k-4}+28.553(7.624)^{k-4}+\frac{1.910}{k}(205.760^{k}-25.515^{k})
\end{align}
and 
\begin{align*}
H_{k}(T)&:= (0.611\,T^{1/6}\log T\big)^{\,k-4}\Bigg(\frac{1}{10\pi^{3}}\log^{5}(T/2)+3.177\log^{4}(T/2)\notag\\&\quad+12.644\log^{3}(T/2)-243234.568
\Bigg).
\end{align*}

\end{lemma}
\begin{proof}
    Let $s=1/2+it$, then $\text{d}s=i\text{d}t$. We have
    \begin{align}{\label{fbi}}
        |I(x,T)|& \le \frac{x^{1/2}}{2\pi}\int_{-T}^{T}  \frac{|\zeta^{k}(1/2+it)|}{\sqrt{1/4+t^2}}\, \text{d}t=\frac{2x^{1/2}}{\pi}\int_{ 0}^{T}  \frac{|\zeta^{k}(1/2+it)|}{\sqrt{4t^{2}+1}}\,  \text{d}t.
    \end{align}
We use the following pointwise bounds for $\zeta(1/2+it)$ from \cite[Theorem 1.1]{HPY2022} and \cite[Theorem 1.1(a)]{R2026} to estimate the integral on the right-hand side:
     \begin{align}
|\zeta(1/2+it)|&\le 1.461,\ \ 0\le t\le 3,\label{ehy}\\
  |\zeta(1/2+it)|&\le 0.595t^{1/6}\log t,\ \ 3\le t< 200,\label{bbc}\\
  |\zeta(1/2+it)|&\le 0.592t^{1/6}\log t,\ \ 200\le t<5.5\cdot 10^{7},\label{gbc}\\
  |\zeta(1/2+it)|&\le 0.611t^{1/6}\log t,\ \ 5.5\cdot 10^{7}\le t< 8.97\cdot 10^{17}.\label{onua}
\end{align}
    Now, we split the integration range into four parts, $I_{1}, I_{2}, I_{3},I_{4}$, according to the different bounds available for $\zeta(1/2+it)$ in \cite[Theorem 1.1]{HPY2022} as follows: 
\begin{align}\label{aca}
\frac{2}{\pi}\int_{ 0}^{T} \frac{|\zeta^{k}(1/2+it)|}{\sqrt{4t^{2}+1}}\text{d}t&=\frac{2}{\pi}\left(\int_{ 0}^{3}+\int_{ 3}^{200}+\int_{ 200}^{5.5\cdot 10^{7}}+\int_{ 5.5\cdot 10^{7}}^{T}  \right)\frac{|\zeta^{k}(1/2+it)|}{\sqrt{4t^{2}+1}}\, \text{d}t\notag\\
       &:= I_{1}+I_{2}+I_{3}+I_{4}(T). 
       \end{align} 
      Next, we estimate these integrals.
      
      To estimate $I_{1}$ and $I_{2}$, we apply \eqref{ehy} and
      \eqref{bbc} respectively to estimate $|\zeta(1/2+it)|^{k-4}$ and compute the integral using \textit{Mathematica} version 12.0.0.0. That is,
\begin{align}\label{ga}
        I_{1}&= \frac{2}{\pi}\int_{0}^{3}  \frac{\left|\zeta^{k}(1/2+it)\right|}{\sqrt{4t^{2}+1}}\text{d}t\le \frac{2}{\pi}\left( \max_{t \in [0,\,3]} |\zeta(1/2+it)|^{k-4} \int_{0}^{3}\frac{|\zeta(1/2+it)|^4}{\sqrt{4t^{2}+1}} \, \text{d}t\right)\notag\\
        &< 1.039(1.461)^{k-4}.
        \end{align}
        Similarly, we have
    \begin{align}
    I_{2}&=4\int_{3}^{200} \frac{\left|\zeta^{k}(1/2+it)\right|}{\sqrt{4t^{2}+1}} \, \text{d}t\le4\left(\max_{t \in [3,\,200]} |\zeta(1/2+it)|^{k-4} \int_3^{200} \frac{|\zeta(1/2+it)|^4}{\sqrt{4t^{2}+1}} \, \text{d}t\right)\notag\\
        &< 28.553(7.624)^{k-4}\label{gij}.
        \end{align}
 For $I_{3}$, we apply \eqref{gbc} directly. This gives  
\begin{align}\label{ewe}  
I_{3}&=\frac{2}{\pi}\int_{200}^{5.5\cdot 10^{7}} \frac{\left|\zeta^{k}(1/2+it)\right|}{\sqrt{4t^{2}+1}}\,\text{d}t\le\frac{1}{\pi}(0.592\log{(5.5\cdot10^7)})^{k}\int_{200}^{5.5\cdot10^{7}} t^{k/6-1}\, \text{d}t\notag\\
&< \frac{1.910}{k}(205.760^{k}-25.515^{k}).
        \end{align}
        
 Finally, we estimate $I_{4}$ as follows. We have
\begin{align}\label{vbd}
I_{4}(T)&=\frac{2}{\pi}\int_{5.5\cdot 10^{7}}^{T} \frac{|\zeta^{k}(1/2+it)|}{\sqrt{4t^{2}+1}}\, \text{d}t\notag\\
&\le \frac{2}{\pi}\left(\max_{t\in[5.5\cdot10^{7},\, T]}|\zeta(1/2+it)|^{k-4}\int_{5.5\cdot 10^{7}}^{T} \frac{|\zeta^{4}(1/2+it)|}{\sqrt{4t^{2}+1}}\, \text{d}t\right)\notag\\
&\le \frac{2}{\pi}(0.611T^{1/6}\log T)^{k-4}\int_{5.5\cdot 10^{7}}^{T} \frac{|\zeta^{4}(1/2+it)|}{\sqrt{4t^{2}+1}}\, \text{d}t.
\end{align}
To estimate the integral on the right-hand side, we  define \begin{equation}\label{skel}
\phi_{m}(T):=\int_{0}^{T}|\zeta(1/2+it)|^{2m}\, \text{d}t.
\end{equation}
Applying integration by parts to the integral on the right-hand side of \eqref{vbd} yields 
\begin{equation}\label{vbn}
\frac{2}{\pi}\int_{5.5\cdot 10^{7}}^{T} \frac{|\zeta^{4}(1/2+it)|}{\sqrt{4t^{2}+1}}\, \text{d}t\le \frac{1}{\pi}\left(\frac{\phi_{2}(T)}{T}-\frac{\phi_{2}(5.5\cdot 10^{7})}{5.5\cdot 10^{7}}+\int_{5.5\cdot 10^{7}}^{T}\frac{\phi_{2}(t)}{t^2}\, \text{d}t\right).
\end{equation}
From \cite[Corollary 2]{chourasiya2025explicitforminghamszero}, we have the following result for $\phi_2(t)$ for all $T\ge 10^{7}$: 
\begin{equation}\label{ggy}
    \phi_{2}(T)\le \frac{1}{2\pi^{2}}T\log^{4}(T/2)+ 39.720T\log^{3}(T/2)+5.817\cdot10^{12}.
\end{equation}
 To evaluate the right-hand side of \eqref{vbn}, we make use of the bound provided in \eqref{ggy}. This gives
\begin{align}\label{erb}
\frac{2}{\pi}\int_{5.5\cdot 10^{7}}^{T} \frac{|\zeta^{4}(1/2+it)|}{\sqrt{4t^{2}+1}}\, \text{d}t&<\frac{1}{10\pi^{3}}\log^{5}(T/2)+3.177\log^{4}(T/2)+12.644\log^{3}(T/2)\notag\\
&\quad-243234.568.
\end{align}
Substituting \eqref{erb} into \eqref{vbd} gives
\begin{align}
I_{4}(T) &< \big(0.611\,T^{1/6}\log T\big)^{\,k-4}\Bigg(\frac{1}{10\pi^{3}}\log^{5}(T/2)+3.177\log^{4}(T/2)\notag\\&\quad+12.644\log^{3}(T/2)-243234.568
\Bigg).\label{ant}
\end{align}
Finally, substituting \eqref{ga},
\eqref{gij}, \eqref{ewe}, and \eqref{ant} into \eqref{aca} and the result into \eqref{fbi} completes the proof.
\end{proof}
\begin{remark}
We note that the lower bound for $\phi_{2}(T)$ in \cite[Corollary 2]{chourasiya2025explicitforminghamszero} remains negative up to $T \geq 1.024\cdot 10^{342}$, which is far beyond the range relevant to our analysis.  Hence, we employ only the bound \eqref{ggy} to estimate all the terms on the right-hand side of \eqref{vbn}.
\end{remark}

\begin{remark}
\label{rkm:sharperRiemann}
 Sharper explicit bounds for $|\zeta(1/2+it)|$ exist for large $t$. For instance, by \cite[Theorem 1.1(b)]{R2026} we have $|\zeta(1/2+it)|\le 0.566t^{1/6}\log{t}$ for $t\ge 8.97\cdot 10^{17}$, and by \cite[Theorem 1.1]{PY2024}, $|\zeta(1/2+it)|\le 66.7t^{27/164}$ for $t\ge3$, which is sharper only for $t\gtrsim 1.1\cdot 10^{42}$. Since larger lower bounds for $T$ force larger admissible lower bounds for $x$ in our main theorem (depending on $k$, $\varepsilon, \varepsilon_1$ and the moment results we used, see Section \ref{sec:Effective}), we choose $T$ differently in each case. For the second moment (i.e. Lemma \ref{lemmaMM21}) we take $T\ge 1.1\cdot 10^{30}$ since sharper explicit second moment estimates are available beyond this range and the choices of $\varepsilon$ and $\varepsilon_{1}$ in Theorem \ref{thm:main} already force the corresponding lower bound for $x$ to be large. In the fourth moment case, sharper estimates at lower heights (see \cite[Corollary 2]{chourasiya2025explicitforminghamszero}) allow smaller $x$, so we avoid imposing large lower bounds on $T$.
\end{remark}

\begin{lemma}
\label{lemma:fourthmomentExt}
    Suppose that $x>1$, $\sigma=1/2$, and $T \geq 3000$. For any positive integer $k\ge 4$ we have 
    \begin{equation*}
|F(x,T)|:=\left|\frac{1}{2\pi i}\int_{\sigma - iT}^{\sigma + iT} \zeta^{k}(s) \frac{x^s}{s} \mathrm{d}s\right|< x^{1/2}[S_{k}(T)+l_{k}],
\end{equation*}
where the constant
\begin{equation}
\label{def:l}
l_{k}:=1.039(1.461)^{k-4}+28.553(7.624)^{k-4}+194.306(18.001)^{k-4}.
\end{equation}
and 
\begin{align*}
S_{k}(T) &:= (0.611T^{1/6}\log T)^{k-4}\Bigg(\frac{1}{5\pi^{3}}\log^{5}(T/2)+1.466\log^{4.5}(T/2)+\frac{1}{\pi^{3}}\log^4(T/2) \notag\\&\quad+6.597\log^{3.5}(T/2)-11266.536\Bigg).
\end{align*}
\end{lemma}

\begin{proof}
    We proceed as in the proof of Lemma \ref{lemma:fourthmomentEst}, splitting $0\le t\le T$ as follows:
    \begin{align*}
\frac{2}{\pi}\int_{ 0}^{T} \frac{|\zeta^{k}(1/2+it)|}{\sqrt{4t^{2}+1}}\,\text{d}t&=\frac{2}{\pi}\left(\int_{ 0}^{3}+\int_{ 3}^{200}+\int_{ 200}^{3000}+\int_{3000}^{T}  \right)\frac{|\zeta^{k}(1/2+it)|}{\sqrt{4t^{2}+1}}\, \text{d}t\notag\\
       &:= F_{1}+F_{2}+F_{3}+F_{4}(T). 
       \end{align*} 
The upper bound estimates for $F_{1}$ and $F_{2}$ are the same as what we have for $I_{1}$ and $I_{2}$ respectively. We estimate $F_{3}$ and $F_{4}(T)$ using the same argument used in the   estimation of $I_{4}(T)$ in Lemma \ref{lemma:fourthmomentEst}. For $ F_{3}$, we apply the following estimates 
\begin{align*}
    \phi_{2}(T)&\leq \frac{T\log^{4}(T/2)}{2\pi^{2}}+0.304T\log^{3}(T/2)+2.100
\end{align*}
and 
\begin{align*}
  \phi_{2}(T)&\geq \frac{T\log^{4}(T/2)}{2\pi^{2}}+0.049T\log^{3}(T/2)+1.968.
\end{align*}
for $3\le T\le 3000$ in \cite{chourasiya2025explicitforminghamszero} and \eqref{gbc}. This gives
\begin{align*}
    F_{3}&\le 194.306(18.001)^{k-4}.
\end{align*}
Now, we use \cite[Corollary 2]{chourasiya2025explicitforminghamszero} and \eqref{onua} to estimate $F_{4}(T)$. This yields
\begin{align*}
   F_{4}(T)&< (0.611T^{1/6}\log T)^{k-4}\Bigg(\frac{1}{5\pi^{3}}\log^{5}(T/2)+1.466\log^{4.5}(T/2)\\
   &\quad+\frac{1}{\pi^{3}}\log^4(T/2) +6.597\log^{3.5}(T/2)-11266.536\Bigg).
\end{align*}
The proof is complete.
\end{proof}

In the next set of lemmas, we estimate the integral using existing explicit bounds on the second moment of $\zeta(1/2+it)$. Here, we consider different ranges of $T$. We first consider the case where $T$ is large.

\begin{lemma}\label{lemmaMM21}
 Let $x>1$, $\sigma=1/2$, and $T \geq 1.1\cdot10^{30}$. For any positive integer $k\ge 2$ we have 
    \begin{equation*}
|J(x,T)|:=\left|\frac{1}{2\pi i}\int_{\sigma - iT}^{\sigma + iT} \zeta^{k}(s) \frac{x^s}{s}\, \mathrm{d}s\right|< x^{1/2}[G_{k}(T)+b_{k}],
\end{equation*}
    where the constant
    
\begin{align}
b_k &:= \frac{1.910}{k} \Bigl(
    (3.978\cdot 10^6)^k + 24803.958^k + 205.760^k - 38448.929^k - 492.548^k\notag \\ 
    &\qquad- 25.515^k
\Bigr)+ 0.748 \cdot 1.461^{\,k-2} + 3.329 \cdot 7.624^{k-2}\label{def:bk}
\end{align}

    and
\begin{align}\label{nega39}
G_{k}(T)&:=(0.566\,T^{1/6}\log T)^{k-2}\Bigg(
      \frac{\log^{2} T}{2\pi}+\frac{3v(\log{T})^{5/3}}{2(1.1\cdot 10^{30})^{2/3}\pi}-0.217\log{T}\notag\\&\quad-\frac{v\log^{5/3}T}{2\pi T^{2/3}}
-741.434
\Bigg),
\end{align}
where $\upsilon := 1271506.721$.
\end{lemma}

\begin{proof}
The proof follows the preceding lemma, replacing the fourth moment with the second. Hence, we have
\begin{align*}
       \frac{2}{\pi}\int_{ 0}^{T} \frac{|\zeta^{k}(1/2+it)|}{\sqrt{4t^{2}+1}}\,\text{d}t&=\frac{2}{\pi}\left(\int_{ 0}^{3}+\int_{ 3}^{200}+\int_{200}^{ 1.1\cdot 10^{30}}+\int_{1.1\cdot10^{30}}^{T}  \right)\frac{|\zeta^{k}(1/2+it)|}{\sqrt{4t^{2}+1}}\, \text{d}t\\
       &:= J_{1}+J_{2}+J_{3}+J_{4}(T). 
       \end{align*}
     The estimates for $J_{1}$ and $J_{2}$ are given by
\begin{equation*}
    J_{1} \leq 0.748 (1.461)^{k-2} \quad\text{and}\quad J_{2} \leq 3.329(7.624)^{k-2}.
\end{equation*}
Note that by  \cite[Theorem 1.1(b)]{R2026} we also have
 \begin{equation}
     |\zeta(1/2+it)|\le 0.566t^{1/6}\log t,\ \ t\ge 8.97\cdot 10^{17}.\label{fnaa}
 \end{equation}
Hence, we estimate $J_{3}$ using \eqref{gbc}, \eqref{onua} and  \eqref{fnaa}. That is
\begin{align*}
J_{3}&=\frac{2}{\pi}\left(\int_{200}^{5.5\cdot10^{7}}+\int_{5.5\cdot10^{7}}^{8.97\cdot10^{17}}+\int_{8.97\cdot10^{17}}^{1.1\cdot10^{30}}\right)\frac{|\zeta^{k}(1/2+it)|}{\sqrt{4t^{2}+1}}\, \text{d}t\\
&\le \frac{1.910}{k} \Bigl(
    (3.978\cdot 10^6)^k + 24803.958^k + 205.760^k \\ 
    &\qquad - 38448.929^k- 492.548^k- 25.515^k
\Bigr).
\end{align*}

Now, we proceed with the estimation of $J_{4}(T)$. We have
\begin{align}
J_{4}(T)&=\frac{2}{\pi}\int_{1.1\cdot10^{30}}^{T}\frac{|\zeta^{k}(1/2+it)|}{\sqrt{4t^{2}+1}}\,\text{d}t\notag\\
&\le \frac{2}{\pi}\left(\max_{t\in[ 1.1\cdot10^{30},\, T]}|\zeta(1/2+it)|^{k-2}\int_{1.1\cdot10^{30}}^{T}\frac{|\zeta^{2}(1/2+it)|}{\sqrt{4t^{2}+1}}\,\text{d}t\right)\notag\\
&\le \frac{2}{\pi}(0.566T^{1/6}\log T)^{k-2}\int_{1.1\cdot10^{30}}^{T}\frac{|\zeta^{2}(1/2+it)|}{\sqrt{4t^{2}+1}}\,\text{d}t\label{fhf}.
\end{align}
Applying integration by parts, the integral on the right-hand side becomes 
\begin{align*}
\frac{2}{\pi}\int_{1.1\cdot10^{30}}^{T}\frac{|\zeta^{2}(1/2+it)|}{\sqrt{4t^{2}+1}}\,\text{d}t\le \frac{1}{\pi}\left(\frac{\phi_{1}(T)}{T}-\frac{\phi_{1}(1.1\cdot 10^{30})}{(1.1\cdot 10^{30})}+\int_{1.1\cdot 10^{30}}^{T}\frac{\phi_{1}(t)}{t^{2}}\, \text{d}t\right),
\end{align*}
where $\phi_{1}(T)$ denotes the case $m=1$ of the definition in \eqref{skel}.

We use \cite[Theorem~3]{MR4500746}, which for all $T \geq 1.1 \cdot 10^{30}$ gives the bound 
\begin{equation*} 
\phi_{1}(T) \leq T \log T-(1+\log{(2\pi)} - 2\gamma)T + \upsilon T^{1/3}\log^{5/3} T, 
\end{equation*}
to estimate the first and last term on the right-hand side, and the lower bound on $\phi_{1}(t)$ in \cite[Theorem 4.3]{JTNB} to estimate the middle term. In particular, the inequality
\begin{equation*}
    \phi_{1}(T) \geq T\log T-2T\sqrt{\log T}-0.93213T+1.242.
\end{equation*}
This results in
\begin{align}\label{ert}
\frac{2}{\pi}\int_{1.1\cdot 10^{30}}^{T} \frac{|\zeta^{2}(1/2+it)|}{\sqrt{4t^{2}+1}}\, \text{d}t&<\frac{\log^{2} T}{2\pi}+\frac{3v(\log{T})^{5/3}}{2(1.1\cdot10^{30})^{2/3}\pi}-0.217\log{T}-\frac{v\log^{5/3}T}{2\pi T^{2/3}}\notag\\&\quad
-741.434.
\end{align}
Substituting \eqref{ert} into \eqref{fhf} gives
\begin{align*}
J_{4}(T)<&(0.566\,T^{1/6}\log T)^{\,k-2}\Bigg(
      \frac{\log^{2} T}{2\pi}+\frac{3v(\log{T})^{5/3}}{2(1.1\cdot10^{30})^{2/3}\pi}-0.217\log{T}\\&\quad-\frac{v\log^{5/3}T}{2\pi T^{2/3}}
-741.434
\Bigg).
\end{align*}
 This completes the proof.
\end{proof}

We now turn to the case where $T$ is small, and establish the following lemma.
\begin{lemma}\label{lemmaMM2}
     Suppose that $x>1$, $\sigma=1/2$, and $T \geq 4$, then for any positive integer $k\ge 2$ we have
     \begin{equation*}
         |L(x,T)|:= \left|\frac{1}{2\pi i}\int_{\sigma - iT}^{\sigma + iT} \zeta^{k}(s) \frac{x^s}{s}\, \mathrm{d}s\right|< x^{1/2}[V_k(T)+u_k],
     \end{equation*}
     where the constant 
     \begin{equation}
    \label{def:uk}
         u_{k}:=0.748(1.461)^{k-2}+0.030(1.040)^{k-2}
     \end{equation}
     and 
     \begin{align}
    V_k(T)&:=(0.611\,T^{1/6}\log T)^{k-2}\Bigg(\frac{\log^{2}T}{2\pi}+0.425\log^{3/2} T+7.658\log T \label{def:Vk}\\
    &\quad +0.637\sqrt{\log T}-0.458\Bigg).\nonumber
     \end{align}
\end{lemma}
\begin{proof}
 Similar to the proof of the preceding lemma, we have
\begin{align*}
       \frac{2}{\pi}\int_{ 0}^{T} \frac{|\zeta^{k}(1/2+it)|}{\sqrt{4t^{2}+1}}\,\text{d}t&=\frac{2}{\pi}\left(\int_{ 0}^{3}+\int_{ 3}^{4}+\int_{4}^{T}  \right)\frac{|\zeta^{k}(1/2+it)|}{\sqrt{4t^{2}+1}}\, \text{d}t\\
       &:= S_{1}+S_{2}+S_{3}(T). 
       \end{align*}
     The estimates for $S_{1}$ and $S_{2}$ are given by
\begin{equation*}
    S_{1} \le 0.748 (1.461)^{k-2} \quad\text{and}\quad S_{2} \leq 0.030(1.040)^{k-2}.
\end{equation*}
In \cite[Theorem 4.3]{JTNB}, an estimate for $\phi_{1}(t)$ is given for $t\in[1,\,T]$. To apply this, we first compute $\phi_{1}(t)$ for $t\in[0,\,1]$. This yields the following inequalities for $t\in[0,\,T]$: 
\begin{align*} 
\phi_{1}(T) &\leq T \log T+2T\sqrt{\log T}+ 23.05779T+1.243,\\
\phi_{1}(T) &\geq T\log T-2T\sqrt{\log T}-0.93213T+1.242.
\end{align*}
Applying these bounds to estimate $S_{3}(T)$ gives
\begin{align*}
  S_{3}(T)&< (0.611\,T^{1/6}\log T)^{k-2}\Bigg(\frac{\log^{2}T}{2\pi}+0.425\log^{3/2} T+7.658\log T\\&\quad+0.637\sqrt{\log T}-0.458\Bigg). 
\end{align*}
The proof is complete.
\end{proof}

\subsection{Integral bound via the functional equation}
\label{sec:nonMoment}

\begin{lemma}\label{nnd}
    Denote $s:=\sigma+it$ and assume that $-1/2\leq \sigma<0$ and $t \geq 1$. Then we have
    \begin{align*}
        & \pi^{s-1/2}\frac{\Gamma(\frac{1-s}{2})}{\Gamma(\frac{s}{2})}=\left(\frac{(1-\sigma)^2+t^2}{4}\right)^{-\sigma/4}\left(\frac{\sigma^2+t^2}{4}\right)^{(1-\sigma)/4}\pi^{\sigma-\frac{1}{2}} \\
         &\quad\cdot \exp\left(\sigma-1/2+it\left(-\log{t}+\log{(2\pi)}+1\right)+E(\sigma,t)\right),
    \end{align*}
    where
    \begin{align*}
       &\left|E(\sigma, t)\right|\leq \frac{\pi}{4}+\frac{1-2\sigma}{2}+\frac{18\sigma^2-18\sigma+19}{12t}.
    \end{align*}
\end{lemma}

\begin{proof}
Note that we can write
    $
        \pi^{s-1/2}=\pi^{\sigma-1/2}\exp(it\log{\pi}).
    $
Hence, we concentrate on estimating the ratio of Gamma-functions. 

By applying \cite[Theorem]{Spira1971}, we can write
    \begin{align}
        & \log{\Gamma\left(\frac{1-s}{2}\right)}-\log{\Gamma\left(\frac{s}{2}\right)}= -\frac{s}{2}\left(\log{\frac{1-s}{2}}+\log{\frac{s}{2}}\right)+\frac{1}{2}\log{\frac{s}{2}} \label{eq:moreCOmplicated} \\
        &\quad +s-\frac{1}{2}+\frac{1}{6}\left(\frac{1}{1-s}-\frac{1}{s}\right)+E_1(s), \label{eq:RealsGammas}
    \end{align}
    where $|E_1(s)|\leq (3|1-s|)^{-1}+2/(3t)$. The term $s-1/2=\sigma+it-1/2$ is easy to add to the main result.
    The absolute value of the other terms than $s-1/2$ in line \eqref{eq:RealsGammas} can be easily estimated as
    \begin{equation*}
        \leq \frac{1}{6|1-s|}+\frac{1}{6|s|}+|E_1(s)| <\frac{1}{3|s|}+\frac{1}{t} <\frac{4}{3t}.
    \end{equation*}
    
    Let us now consider real parts of the terms in line \eqref{eq:moreCOmplicated}. We have
    \begin{equation*}
       \Re\left( \frac{1}{2}\log{\frac{s}{2}}\right)=\frac{1}{2}\log{\left|\frac{s}{2}\right|}=\frac{1}{4}\log{\left(\frac{\sigma^2+t^2}{4}\right)}
    \end{equation*}
    and
    \begin{align*}
       & \Re\left(-\frac{s}{2}\left(\log{\frac{1-s}{2}}+\log{\frac{s}{2}}\right)\right) \\
       &\quad= -\frac{\sigma}{2}\left(\log{\left|\frac{1-s}{2}\right|}+\log{\left|\frac{s}{2}\right|}\right)+\frac{t}{2}\left(\text{Arg}\left(\frac{1-s}{2}\right)+\text{Arg}\left(\frac{s}{2}\right)\right) \\
       &\quad = -\frac{\sigma}{4}\left(\log{\left(\frac{(1-\sigma)^2+t^2}{4}\right)}+\log{\left(\frac{\sigma^2+t^2}{4}\right)}\right)-\frac{t}{2}\,\arctan{\left(\frac{t}{1-\sigma}\right)}\\
       &\quad\quad+\frac{t}{2}\left(\arctan\left(\frac{t}{\sigma}\right)+\pi\right).
    \end{align*}
    Now we can apply the estimate
    \begin{equation}
    \label{eq:arctanBounds}
        \frac{\pi}{2}-\frac{1}{y}<\arctan(y)<\frac{\pi}{2}-\frac{1}{y}+\frac{1}{3y^2} \text{ for } y>0
    \end{equation}
    for the arctangents. Hence, the absolute value of their contribution is smaller than 
    \begin{equation*}
        \max\left\{\frac{1-2\sigma}{2},\left|\frac{1-2\sigma}{2}-\frac{1-2\sigma}{6t}\right|\right\}=\frac{1-2\sigma}{2}
    \end{equation*}
    since $t\geq 1$ and $-1/2 \leq \sigma <0$. Thus, we have estimated the real parts of \eqref{eq:moreCOmplicated}.

    Let us now consider imaginary parts of \eqref{eq:moreCOmplicated}.
    Similarly as we concluded with the real parts, keeping in mind that $-1/2\leq \sigma <0$ and $t \geq 1$, we have
    \begin{equation*}
     0< \Im\left(\frac{1}{2}\log{\frac{s}{2}}\right)=\frac{1}{2}\text{Arg}\left(\frac{s}{2}\right)\leq \frac{\pi}{2}+\frac{1}{2}\left(-\frac{\sigma}{t}-\frac{\pi}{2}\right)=\frac{\pi}{4}-\frac{\sigma}{2t}.
    \end{equation*}
    We can write 
    \begin{align*}
        &\Im\left(-\frac{s}{2}\left(\log{\frac{1-s}{2}}+\log{\frac{s}{2}}\right)\right) \\
        &\quad =-\frac{t}{2}\left(2\log{t}+\frac{1}{2}\log{\left(\left(\frac{1-\sigma}{t}\right)^2+1\right)}+\frac{1}{2}\log{\left(\frac{\sigma^2}{t^2}+1\right)}-2\log{2}\right) \\
        &\quad\quad-\frac{\sigma}{2}\left(-\arctan{\left(\frac{t}{1-\sigma}\right)}+\arctan\left(\frac{t}{\sigma}\right)+\pi\right).
    \end{align*}
    We can again apply estimates \eqref{eq:arctanBounds}, and conclude that
    \begin{equation*}
       \left|-\arctan{\left(\frac{t}{1-\sigma}\right)}+\arctan\left(\frac{t}{\sigma}\right)+\pi\right|< \frac{1-2\sigma}{t}.
    \end{equation*}
    Lastly, we use the fact that $\log{(1+x)}\leq x$ to obtain 
    \begin{equation*}
        \left|-\frac{t}{2}\left(\frac{1}{2}\log{\left(\left(\frac{1-\sigma}{t}\right)^2+1\right)}+\frac{1}{2}\log{\left(\frac{\sigma^2}{t^2}+1\right)}\right)\right|\leq \frac{(1-\sigma)^2+\sigma^2}{4t}.
    \end{equation*}
    The claim follows when we combine all of the previous estimates. 
\end{proof}

\begin{lemma} 
\label{lemma:integralwithoutMoments}
 Let $k\ge 2$, $-0.31< b< 0$, and $T \geq 3$. Then we have
 \begin{equation*}
 \left|\frac{1}{2\pi i}\int_{b - iT}^{b + iT} \frac{\zeta^{k}(s)x^{s}}{s}\,\mathrm{d}s\right|< x^{b}(B(b,T)+A(b, \, k)),
\end{equation*}
where
\begin{align*}
&B(b,T):=\frac{8C_{k,1}(b)\zeta^{k}(1-b)}{\pi\sqrt{k}}\left(\frac{(1-b)^2+9}{36}\right)^{-\frac{kb}{4}}\left(\frac{b^2+9}{36}\right)^{\frac{k(1-b)}{4}}\\
&\quad \cdot \exp\left(k\left(\frac{18b^2-18b+19}{36}\right)\right)T^{k(1/2 - b)-\frac{1}{2}} 
   \left(1 + \frac{|b|}{T}
   \right)
\end{align*}
with the constants
\begin{equation}\label{consk}
    C_{k,1}(b):=(\pi e)^{k(b-1/2)}\exp\left(\frac{k\pi}{4}+\frac{k(1-2b)}{2}\right),
\end{equation}
\begin{equation*}
 A(b, \, k):=\frac{g_k(b)}{\pi}\zeta^{k}(1-b)
\end{equation*}
and 
\begin{equation*}
   g_k(b):= \int_{0}^{3}\frac{|\chi^{k}(b+it)|}{\sqrt{b^{2}+t^{2}}}\,\mathrm  {d}t.
\end{equation*}
\end{lemma}
\begin{proof}
   For $s=b+it$, the functional equation of the Riemann zeta function is given by $\zeta(s)=\chi(s)\zeta(1-s)$, where from \cite[Equation 2.1.10]{MR882550}
   \begin{equation}\label{kaka}
\chi(s)=\pi^{s-1/2}\frac{\Gamma\left(\frac{1-s}{2}\right)}{\Gamma\left(\frac{s}{2}\right)}.
\end{equation}
Hence, we have
   \begin{align} \label{bbtv}
   \left|\frac{1}{2\pi i}\int_{b - iT}^{b + iT} \frac{\zeta^{k}(s)x^{s}}{s}\,\text{d}s\right|
&=\frac{x^{b}}{\pi}\sum_{n=1}^{\infty}\frac{d_{k}(n)}{n^{1-b}}\left|\int_{0 }^{T}\frac{\chi^{k}(b+it)}{b+it}(nx)^{it}\,\text{d}t\right|.
   \end{align}
We split the integral on the right-hand side as follows:
\begin{equation}\label{agya}
\left|\int_{0}^{T}\frac{\chi^{k}(b+it)}{b+it}(nx)^{it}\,\text{d}t
\right|\le\left|\int_{0}^{3}\frac{\chi^{k}(b+it)}{b+it}(nx)^{it}\,\text{d}t\right|
+ \left|\int_{3}^{T}\frac{\chi^{k}(b+it)}{b+it}(nx)^{it}\,\text{d}t\right|.
\end{equation}
Estimating the first term on the right-hand side of \eqref{agya} gives
\begin{equation}\label{vw}
\left|\int_{0}^{3}\frac{\chi^{k}(b+it)}{b+it}(nx)^{it}\,\text{d}t\right|\le\int_{0}^{3}\frac{|\chi^{k}(b+it)|}{\sqrt{b^{2}+t^{2}}}\,\text{d}t=g_k(b),
\end{equation}
where $g_{k}(b)$ is a constant that depends on $b$ and $k$.

 Now, substituting \eqref{kaka} into the second term on the right-hand side of \eqref{agya} gives 
 \begin{align}{\label{jah}}
\left|\int_{3}^{T}\frac{\chi^{k}(b+it)}{b +it}(nx)^{it}\,\text{d}t\right|&= \left|\int_{3}^{T}\left(\pi^{s-1/2}\frac{\Gamma\left(\frac{1-s}{2}\right)}{\Gamma\left(\frac{s}{2}\right)}\right)^{k}\cdot\frac{\exp({it\log(nx)})}{b+it}\, \text{d}t\right|.
 \end{align}
 Note that 
 \begin{equation}\label{dave}
    \frac{1}{b +it}=\frac{b}{(b^{2}+t^{2})}-\frac{it}{(b^{2}+t^{2})}.
 \end{equation}
 Substituting the results in Lemma \ref{nnd} and \eqref{dave} into the expression on the right-hand side \eqref{jah} gives 
 \begin{align}\label{ncm} \left|\int_{3}^{T}\left(\pi^{s-1/2}\frac{\Gamma\left(\frac{1-s}{2}\right)}{\Gamma\left(\frac{s}{2}\right)}\right)^{k}\cdot\frac{\exp({it\log(nx)})}{b+it}\, \text{d}t\right|&\le \left|iC_{k,1}\int_{3}^{T}\frac{tW(t)}{b^{2}+t^{2}}e^{iF(t)}\,\text{d}t\right|\notag\\ 
 &\quad +\left|b C_{k,1}\int_{3}^{T}\frac{W(t)}{b^{2}+t^{2}}e^{iF(t)}\,\text{d}t\right|,
 \end{align}
 where the constant $C_{k,1}$ is defined in \eqref{consk},
 \begin{equation*}
 F(t)=kt(-\log t+\log{(2\pi)}+1)+t\log{(nx)},
 \end{equation*}
 and
  \begin{align*}
     W(t)&:=\left(\frac{(1-b)^{2}+t^{2}}{4}\right)^{\frac{-kb}{4}}\left(\frac{b^{2}+t^{2}}{4}\right)^{\frac{k(1-b)}{4}}\exp\left(k\left(\frac{18b^2-18b+19}{12t}\right)\right).
 \end{align*}
 To estimate the integrals on the right-hand side of \eqref{ncm}, we apply \cite[Lemma~4.5]{MR882550} directly since $F(t)$ is twice differentiable and $F''(t)\ge r>0$ in the interval $t\in[1,T]$.  
Thus
\begin{equation*}
    |F''(t)| = \frac{k}{t}\ge \frac{k}{T}.
\end{equation*}
Hence, we choose $r = k/T$. 

To determine 
\begin{equation*}
M := \left|\frac{tW(t)}{b^{2}+
t^{2}}\right|    
\end{equation*}
 from the first term on the right-hand side of \eqref{ncm}, we first note that, for $- 0.31 <  b < 0$ and $t \in [3, T]$, the following inequalities hold:
\begin{equation*}
(1 - b)^{2} + t^{2} \leq \tfrac{(1 - b)^{2}+9}{9}t^{2}
\quad \text{and} \quad
t^{2}< b^{2} + t^{2} \leq \tfrac{b^{2}+9}{9}t^{2}.   
\end{equation*}
Also, we have 
\begin{equation*}
    \exp\left(k\left(\frac{18b^2-18b+19}{12t}\right)\right) \leq \exp\left(k\left(\frac{18b^2-18b+19}{36}\right)\right).
\end{equation*}
This together gives
\begin{equation*}
    M\leq\left(\frac{(1-b)^2+9}{36}\right)^{-\frac{kb}{4}}\left(\frac{b^2+9}{36}\right)^{\frac{k(1-b)}{4}}\exp\left(k\left(\frac{18b^2-18b+19}{36}\right)\right)T^{k(1/2-b)-1}.
\end{equation*}
Therefore, we have
\begin{align}\label{mmkk}
    &\left|\int_{3}^{T}\frac{tW(t)}{b^{2}+t^{2}}e^{iF(t)}\,\text{d}t\right|\le\frac{8}{\sqrt{k}}\left(\frac{(1-b)^2+9}{36}\right)^{-\frac{kb}{4}}\left(\frac{b^2+9}{36}\right)^{\frac{k(1-b)}{4}} \notag\\
    &\quad\cdot\exp\left(k\left(\frac{18b^2-18b+19}{36}\right)\right)T^{k(1/2-b)-\frac{1}{2}}.
\end{align}
Similarly, the integral in the second term on the right-hand side of \eqref{ncm} becomes
\begin{align}\label{nnrr}
     &\left|\int_{3}^{T}\frac{W(t)}{b^{2}+t^{2}}e^{iF(t)}\,\text{d}t\right|\le\frac{8}{\sqrt{k}}\left(\frac{(1-b)^2+9}{36}\right)^{-\frac{kb}{4}}\left(\frac{b^2+9}{36}\right)^{\frac{k(1-b)}{4}} \notag\\
    &\quad\cdot\exp\left(k\left(\frac{18b^2-18b+19}{36}\right)\right)T^{k(1/2-b)-\frac{3}{2}}.
\end{align}
Substituting \eqref{mmkk} and \eqref{nnrr} into \eqref{ncm} and simplifying gives
\begin{align}\label{acada}
&\left|
    \int_{1}^{T}\left(\pi^{\,s-\frac12}
    \frac{\Gamma\!\left(\frac{1-s}{2}\right)}{\Gamma\!\left(\frac{s}{2}\right)}
        \right)^{\!k}
        \frac{e^{\,it\log(nx)}}{b + it}\, \text{d}t\right|
\le \frac{8C_{k,1}}{\sqrt{k}}\exp\left(k\left(\frac{18b^2-18b+19}{36}\right)\right)
   \notag\\
&\quad \cdot \left(\frac{(1-b)^2+9}{36}\right)^{-\frac{kb}{4}}\left(\frac{b^2+9}{36}\right)^{\frac{k(1-b)}{4}}T^{k(1/2 - b)-\frac{1}{2}} 
   \left(1+\frac{|b|}{T}
   \right).
\end{align}

Finally, substituting \eqref{vw} and \eqref{acada} into \eqref{agya}, and subsequently inserting the resulting expression into \eqref{bbtv} yields the desired bound.
\end{proof}

\section{Effective results}
\label{sec:Effective}
In this section, we provide effective results that can be used to derive numerical estimates for different $k$, $\varepsilon$ and lower bounds for $x$. We have different cases depending on what method is used to derive estimates for the integral $\int_{\sigma - iT}^{\sigma + iT} \zeta^{k}(s) \frac{x^s}{s} \text{d}s$, where $\sigma=1/2$ or $\sigma<0$ (see Sections \ref{sec:momentEstimates} and \ref{sec:nonMoment}), and whether $x$ is an integer or not. 

\subsection{Cases where $x$ is not an integer or a half-integer}
In Lemma \ref{lemma:dkGeneral}, we divided by the term $|x-N|$, where $N$ is the closest integer to $x$, when $x$ is not an integer. It is worth noting that the term $x-N$ is very close to zero when $x$ is very close to an integer. This means that we are multiplying by a large number. To avoid this problem, we restrict our estimates to two cases: (i) $x$ is a half-integer, and (ii) $x$ is an integer. Then, we derive estimates for or other values $x$ for using the following lemma. To present the results, we define
\begin{align*}
    r_7(\tau,\upsilon, j,x):=& \left(x+\frac{1}{2}\right)^{-\tau}\left(\log{\left(x+\frac{1}{2}\right)}\right)^{-\upsilon}\cdot \\
    &\cdot\begin{cases}
    \sum_{m=1}^{j}\binom{j}{m}\frac{x+1/2}{(2x)^m(\log{(x+1/2)})^{m-1}}, &\text{if } j \geq 1 \\
    \frac{1}{2}, &\text{if } j=0
    \end{cases}
\end{align*}
and
\begin{equation*}
    r_8(\tau, \upsilon, j, x):=
    \begin{cases}
        0, &\text{if } j=0 \\
        \frac{(\log{(x+1/2)})^j}{2(x+1/2)^{\tau}(\log{(x+1/2)})^{\upsilon}}, &\text{if } 1\leq j<\tau\log{\left(x+\frac{1}{2}\right)}+\upsilon \\
         \frac{1}{2}\left(\frac{j-\upsilon}{\tau}\right)^{j-\upsilon}e^{\upsilon-j}, &\text{if } j\geq \tau\log{\left(x+\frac{1}{2}\right)}+\upsilon.
    \end{cases}
\end{equation*}

\begin{lemma}
\label{lemma:estimateBetweenPoints}
    Let $k \geq 2$. Assume that 
    $
        P_k(x)=\sum_{j=0}^{k-1} a_j x^j,
    $
    where $a_{k-1}=1/(k-1)!$, and that
    \begin{equation*}
        \left|\Delta_k(x)\right| \leq
        \begin{cases}
            \rho_{k,1}x^{\tau_{k}}(\log{x})^{\upsilon_{k}}, &\text{if } x \text{ is an integer}, \\
            \rho_{k,2}x^{\tau_{k}}(\log{x})^{\upsilon_{k}}, &\text{if } x \text{ is a half-integer} 
        \end{cases}
    \end{equation*}
    for all $ x\geq x_1>1$ that are integers or half integers. We also suppose that $\rho_{k,1}, \rho_{k,2},\tau_{k},$ are positive real numbers and $\upsilon_k \geq 0$. 
    
    Then for all real numbers $x \geq x_1+1/2$, we have
    \begin{equation*}
        T_k(x)=xP_k(\log{x})+\Delta_{k,\text{new}}(x),
    \end{equation*}
    where
    \begin{align*}
        \left|\Delta_{k, \text{new}}(x)\right| \leq &\left(\max\{\rho_{k,1}, \rho_{k,2}\}+\sum_{j=0}^{k-1}|a_j|\left(r_7(\tau_k,\upsilon_k,j,x_1)+r_8(\tau_k, \upsilon_k, j, x_1)\right)\right) \\
        &\cdot x^{\tau_{k}}(\log{x})^{\upsilon_{k}}.
    \end{align*}
\end{lemma}

\begin{proof}
First, we consider the case $x\in(n, n+1/2)$ for some integer $n \geq x_1$. Note that now
\begin{equation*}
    T_k(x)=T_k(n)=nP_k(\log{n})+\Delta_k(n).
\end{equation*}
We have
\begin{align}
    &\left|x(\log{x})^j-n(\log{n})^j\right|\leq x(\log{x})^j-\left(x-\frac{1}{2}\right)\left(\log{\left(x-\frac{1}{2}\right)}\right)^j \nonumber \\
    &\quad =x(\log{x})^j \left(1-\left(1+\frac{\log{\left(1-\frac{1}{2x}\right)}}{\log{x}}\right)^j\right)+\frac{1}{2}\left(\log{\left(x-\frac{1}{2}\right)}\right)^j. \label{eq:coefficientsMiddle}
\end{align}
By binomial series, the first term in \eqref{eq:coefficientsMiddle} is
\begin{equation*}
    =-x(\log{x})^j\sum_{m=1}^{j}\binom{j}{m}\left(\frac{\log{(1-(2x)^{-1})}}{\log{x}}\right)^m.
\end{equation*}
Since $|\log{(1-(2x)^{-1})}|=|\log{(1+(2x-1)^{-1})}|\leq (2x-1)^{-1}$, the absolute value of the right-hand side is
\begin{equation}
\label{eq:coefficientsFirst}
    \leq x(\log{x})^j\sum_{m=1}^{j}\binom{j}{m}\left(\frac{1}{(2x-1)\log{x}}\right)^m\leq \sum_{m=1}^{j}\binom{j}{m}\frac{x_1+1/2}{(2x_1)^m(\log{(x_1+1/2)})^{m-1}}
\end{equation}
for all $j\geq 1$ and $x\geq x_1+1/2$. This provides an estimate for the first part of \eqref{eq:coefficientsMiddle}, and we obtain the coefficient it contributes to the term by dividing by $(x_1+1/2)^{\tau_k}(\log{(x_1+1/2)})^{\upsilon_k}$.

Finally, we consider the last term in \eqref{eq:coefficientsMiddle}. In the case $j=0$, the term is $1/2$. For all $j \geq 1$, we estimate 
\begin{align}
    &\frac{1}{2x^{\tau_k}(\log{x})^{\upsilon_k}}\left(\log{\left(x-\frac{1}{2}\right)}\right)^j \leq \frac{(\log{x})^j}{2x^{\tau_k}(\log{x})^{\upsilon_k}} \nonumber\\
    &\quad \leq \begin{cases}
    \frac{1}{2}\left(\frac{j-\upsilon_k}{\tau_k}\right)^{j-\upsilon_k}e^{\upsilon_k-j}, &\text{if } e^{\frac{j-\upsilon_k}{\tau_k}}\geq x_1+\frac{1}{2} \\
    \frac{(\log{(x_1+1/2)})^j}{2(x_1+1/2)^{\tau_k}(\log{(x_1+1/2)})^{\upsilon_k}}, &\text{otherwise}
    \end{cases}
    \label{eq:coefficientsLast}
\end{align}
for $x\geq x_1+1/2$. Combining estimates \eqref{eq:coefficientsFirst}, \eqref{eq:coefficientsLast} and the bound $|\Delta_k(x)|\leq \rho_{k,1}x^{\tau_{k}}(\log{x})^{\upsilon_{k}}$, we get a bound for the case $x\in(n, n+1/2)$ for some integer.

The case $x\in(n+1/2, n+1)$, where $n$ is an integer with $n\geq x_1-1/2$ follows similarly. The only essential difference is that we use the bound $|\Delta_k(x)|\leq \rho_{k,2}x^{\tau_{k}}(\log{x})^{\upsilon_{k}}$. Hence, the result follows.
\end{proof}

The coefficients $a_j$ can be determined for any fixed $k$ numerically. Hence, we focus on deriving the results in the cases where $x$ is a half-integer or an integer.

\subsection{Effective results based on the fourth moment}
\label{sec:EffectiveFourth}

First, we will start with the case where we apply Lemma \ref{lemma:fourthmomentEst} for the fourth moment of the Riemann zeta function. Keeping in mind \cite[Theorem 12.3]{MR882550}, we aim to obtain a result that is of size $\Delta_k(x)=O(x^{(k-1)/(k+2)+\varepsilon})$. This should be obtained from the terms that are of sizes $T^{-1}x^c$, $x^{1/2}T^{k\mu-1}$ and $x^{1/2} T^{(k-4)\mu+\varepsilon}$, where $\mu\leq 1/4$ and $c>1$ are chosen appropriately. With this in mind, we derive the following theorem by choosing a suitable $T$ with respect to $x$. 

In order to present our results, we will define $S_{k,\text{new}}$, $r_9$, $r_{10}$, $r_{11}$ and $r_{12}$. Let
\begin{align}
    &S_{k,\text{new}}(T):=(0.611T^{1/6}\log T)^{k-4}\log^{3.5}(T/2)\Bigg(\frac{1}{5\pi^{3}}\log^{1.5}(T/2)+1.466\log(T/2) \notag\\
    &\quad+\frac{1}{\pi^{3}}\log^{0.5}(T/2) +6.597\Bigg), \label{def:SkNew} \\
    &r_9(x,T, k, \varepsilon, \varepsilon_1,a):=\frac{(1+\pi)x^{1+\varepsilon_1}}{\pi T}\left(r_1(x,k,\varepsilon, 1+\varepsilon_1, a) \vphantom{2\left(x-\frac{1}{2}\right)^{-c}}+4x \left(x-\frac{1}{2}\right)^{\varepsilon-1-\varepsilon_1} \right.\notag \\
    &\quad\left.+2x\left(\frac{x}{a}\right)^{\varepsilon-1-\varepsilon_1}r_2(x,a)+(a+1)\left(x+\frac{1}{2}\right)^{\varepsilon-\varepsilon_1}r_3(x,a)\right), \label{def:r10}\\
   & r_{10}(x,T, T_0, T_1, k, c) \label{def:newr10} \\
   &\quad:= \frac{(\log{T_1})^k}{T\pi}\left(c-\frac{1}{2}\right)\max\left\{a_1\left(c,\tfrac{1}{2},T_0\right)^kT^{\frac{k}{6}}x^{\frac{1}{2}},a_1\left(c,c, T_0\right)^k x^c\right\}, \notag \\
    & r_{11}(f,x, T, T_0,  T_1, k, \varepsilon, \varepsilon_1, a):=\frac{r_9(x, T,k, \varepsilon, \varepsilon_1,a)}{x^{\frac{k-1+\varepsilon_1(k-4)}{k+2}}(\log{x})^{\frac{k^2+2k+6}{k+2}}} \label{def:r12} \\
    &\quad +\frac{r_{10}(x, T, T_0,  T_1, k, 1+\varepsilon_1)+x^{\frac{1}{2}}f( T_1)}{x^{\frac{k-1+\varepsilon_1(k-4)}{k+2}}(\log{x})^{\frac{k^2+2k+6}{k+2}}},\notag \\
    &r_{12}(x, T, T_0, k, \varepsilon, \varepsilon_1, a):=\frac{(1+\pi)x^{1+\varepsilon_1}}{\pi T}\left(r_1(x,k,\varepsilon, 1+\varepsilon_1, a) \vphantom{2\left(x-\frac{1}{2}\right)^{-c}} \right. \notag \\
    &\left.\quad+x\left(\frac{x}{a}\right)^{\varepsilon-1-\varepsilon_1}(r_4(x,a)+0.6)+\left(1+\frac{1}{2x}\right)(x+1)^{\varepsilon-\varepsilon_1}\left(r_3(x,a)+\frac{1}{2}\right) \right) \label{def:r13}\\
    &\quad+\frac{x^{\frac{\lambda(k)}{\log\log{x}}}}{2\pi}\left(\pi+\frac{2T_0(1+\varepsilon_1)}{T_0^2-(1+\varepsilon_1)^2}\right),  \notag \\
    \text{and}\notag\\
    &r_{13}(f,x, T, T_0, T_1, k, \varepsilon, \varepsilon_1,a):=\frac{r_{12}(x, T, T_0, k, \varepsilon, \varepsilon_1,a)}{x^{\frac{k-1+\varepsilon_1(k-4)}{k+2}}(\log{x})^{\frac{k^2+2k+6}{k+2}}}\notag \\
     &\quad +\frac{r_{10}(x, T, T_0, T_1, k, 1+\varepsilon_1)+x^{\frac{1}{2}}f(T_1)}{x^{\frac{k-1+\varepsilon_1(k-4)}{k+2}}(\log{x})^{\frac{k^2+2k+6}{k+2}}}, 
\end{align}
where the terms from $r_1$ to $r_4$ are defined as in \eqref{def:r1}--\eqref{def:r6} and $a_1$ as in \eqref{def:a1}. Moreover, let $\text{LW}_{-1}(x)$ denote the Lambert $W$ function with the branch where $\text{LW}_{-1}(x)\leq -1$ and $-1/e\leq x <0$.

Here, we provide an estimate for $\Delta_{k}(x)$ for all $k\ge 4$ when $x$ is small enough.
\begin{theorem}
\label{thm:EffectiveFourth}
  Assume $k\geq 4$ and $a \in [1.6, e^{k-2}]$. Let $\lambda(k)$ be as in Definition \ref{def:lambdak}, $\varepsilon$, and $x_0$ be as in Lemma \ref{lemma:dkGeneral}, $a_1$ as in \eqref{def:a1}, $\varepsilon_1 \in(\varepsilon, \lambda(k)/\log\log{3}-1)$ and $\varepsilon_1 \leq 0.2$. 

   If $x$ is a half-integer or an integer and $x\geq x_1$, then we have 
   \begin{align*}
        &\left|\Delta_k(x)\right|< \omega_1(x_1,k, \varepsilon, \varepsilon_1,a)x^{\frac{k-1+\varepsilon_1(k-4)}{k+2}}(\log{x})^{\frac{k^2+2k+6}{k+2}}, 
    \end{align*}
    where 
    \begin{align*}
        \omega_1(x_1,k, \varepsilon, \varepsilon_1,a)&:=\max\left\{r_{11}(f,x_1, T, T_0,  T_1, k, \varepsilon, \varepsilon_1, a), \right. \\
       &\left.  r_{13}(f,x_1, T, T_0,  T_1, k, \varepsilon, \varepsilon_1,a)  \right\},
    \end{align*}  
 
\begin{align*}
    &f=S_{k,\text{new}}, \text{ } \kappa_1=\left(\frac{5\pi^2(k+2)a_1(1+\varepsilon_1,1+\varepsilon_1,3000)^k}{6\cdot0.611^{k-4}}\right)^{\frac{6}{k+2}}, \text{ } \\
    & \kappa_2=-\frac{6}{k+2}, \text{ }\kappa_3=\frac{3(1+2\varepsilon_1)}{k+2}, \text{ }T=\kappa_1(\log{x_1})^{\kappa_2}x_1^{\kappa_3}, \text{ } T_0=3000, \text{ } T_1=\kappa_1x_1^{\kappa_3}
\end{align*}
    and $x_1\geq \max\{ax_0, e^e+1/2\}$, $x_1>4a$ if
    \begin{equation*}
\left(\frac{\kappa_1}{T_0}\right)^{\frac{k+2}{6}}>\frac{2}{(1+2\varepsilon_1)e},
    \end{equation*}
    and
    \begin{equation*}
        x_1\geq \max\left\{ax_0, e^e+1/2, \exp\left(-\frac{2}{1+2\varepsilon_1}\text{LW}_{-1}\left(-\frac{1+2\varepsilon_1}{2}\left(\frac{\kappa_1}{T_0}\right)^{\frac{k+2}{6}}\right)\right)\right\},
    \end{equation*}
    $x_1>4a$ otherwise. In addition, if
    \begin{equation*}
        \lambda(k) \geq \frac{4(k-1+\varepsilon_1(k-4))}{k+2},
    \end{equation*}
    we also require that
    \begin{equation}
    \label{eq:x0lambda}
        x_1 \geq \exp\left(\exp\left((k+2)\frac{\lambda(k)+\sqrt{\lambda(k)^2-4\lambda(k)\cdot\frac{k-1+\varepsilon_1(k-4)}{k+2}}}{2(k-1+\varepsilon_1(k-4))}\right)\right).
    \end{equation}
\end{theorem}

\begin{proof}
    We begin by first considering the case where $x$ is a half-integer, $\log{T_1} =\log{T}+O_{k}(\log\log{x})$ and $T \geq 3000$ in our Lemmas. By Lemmas \ref{lemma:TkIntegral}, \ref{lemma:dkGeneral}, \ref{lemma:IntegralsToBeEstimated}, \ref{lemma:sigmaChanges} with $b=1/2$ and \ref{lemma:fourthmomentExt}, our largest terms are:
    size $O(x^{1+\varepsilon'}/T)$:
    \begin{align}
        &\frac{1+\pi}{\pi}\cdot\frac{x^c}{T}\left(2\left(x-\frac{1}{2}\right)^{\varepsilon-c}\left(2x-\frac{1}{2}\right) \right. \nonumber \\
       &\quad \left. +\left(2x-\frac{1}{2}\right)\left(\frac{x}{a}\right)^{\varepsilon-c}\log{\left(x\left(1-\frac{1}{a}\right)-\frac{1}{2}\right)} \right. \label{eq:fourthNonIntegerFirst} \\
       &\quad\left.+(a+1)x\left(x+\frac{1}{2}\right)^{\varepsilon-c}\log{x}+r_1(x,k,\varepsilon, c, a)\right), \nonumber
    \end{align}
    sizes $O(T^{k/6-1+\varepsilon'}x^{1/2})$ and $O(T^{-1+\varepsilon'}x^{c})$:
    \begin{equation}
   \label{eq:fourthNonIntegerSecond} 
       \frac{(\log{T})^k}{T\pi}\left(c-\frac{1}{2}\right)\max\left\{a_1\left(c,\tfrac{1}{2}, T_0\right)^kT^{\frac{k}{6}}x^{\frac{1}{2}},a_1\left(c,c, T_0\right)^k x^c\right\}
    \end{equation}
    and size $O(T^{(k-4)/6+\varepsilon'}x^{1/2})$:
    \begin{equation}
    \label{eq:effectiveFourthHalfInteger}
        \frac{x^{\frac{1}{2}}}{5\pi^{3}}(0.611T^{1/6}\log T)^{k-4}\log^{5}(T/2).
    \end{equation}
    We want to set $T=\kappa_1(\log{x})^{\kappa_2}x^{\kappa_3}$, and find terms $\kappa_1,\kappa_2, \kappa_3$ in such a way that our leading coefficient is the optimal one, and indeed show that this happens for $\kappa_i$ introduced in the statement.

    Asymptotically, the largest contributions come from the last term in \eqref{eq:fourthNonIntegerSecond} and from the term \eqref{eq:effectiveFourthHalfInteger} since we must have $c>1+\varepsilon$. Let us set $c=1+\varepsilon_1$ and $T_0=3000$. We want to have
    \begin{equation*}
        \kappa_3\left(\frac{k-4}{6}\right)+\frac{1}{2}=1+\varepsilon_1-\kappa_3 \quad\iff\quad \kappa_3=\frac{3(1+2\varepsilon_1)}{k+2},
    \end{equation*}
    \begin{equation}
    \label{def:kappa2Fourth}
        \kappa_2\left(\frac{k-4}{6}\right)+k+1=k-\kappa_2 \quad\Leftrightarrow \quad \kappa_2=-\frac{6}{k+2}
    \end{equation}
    and
    \begin{multline}
    \label{def:kappa1FourthFirst}
        \frac{0.611^{k-4}}{5 \pi^3} \kappa_1^{\frac{k-4}{6}}\kappa_3^{k+1}=\frac{\frac{1}{2}+\varepsilon_1}{\pi \kappa_1}(a_1(c,c,3000)\kappa_3)^k \\
        \iff \quad \kappa_1=\left(\frac{5\pi^2(1+2\varepsilon_1)a_1(c,c,3000)^k}{2\cdot0.611^{k-4}\kappa_3}\right)^{\frac{6}{k+2}}=\left(\frac{5\pi^2(k+2)a_1(c,c,3000)^k}{6\cdot0.611^{k-4}}\right)^{\frac{6}{k+2}}.
    \end{multline}
    Hence, in our estimate for $|\Delta_k(x)|$, the power of $x$ is 
    \begin{equation*}
        \kappa_3\left(\frac{k-4}{6}\right)+\frac{1}{2}=1+\varepsilon_1-\kappa_3=\frac{k-1+\varepsilon_1(k-4)}{k+2},
    \end{equation*}
    and the power of logarithm is
    \begin{equation*}
        \kappa_2\left(\frac{k-2}{6}\right)+k+1=k-\kappa_2 = \frac{k^2+2k+6}{k+2}.
    \end{equation*}
     Since $\kappa_2<0$, we choose $T_1=\kappa_1x^{\kappa_3}$ to ensure that terms $\log{T_1}/\log{x}$ are decreasing for all $x$. Next, we derive the constant $\omega_1$ in front of estimate of $|\Delta_k(x)|$.
    
    The idea is to show that if $x$ is a half-integer, then the constant in front of estimate of $|\Delta_k(x)|$ is at most $r_{11}$ (see \eqref{def:r12}). Moreover, we will conclude that $r_{11}$ is a decreasing function with respect to $x$ and hence we can set $x=x_1$ in $r_{11}$. Let us first consider the terms coming from Lemma \ref{lemma:dkGeneral}. We start by noting that we have
    \begin{equation}
    \label{eq:r1inr10}
        \frac{x^c r_1(x,k,\varepsilon, c,a)}{T\cdot x^{c-\kappa_3}(\log{x})^{k-\kappa_2}}=\frac{r_1(x,k,\varepsilon, c,a)}{\kappa_1(\log{x})^k}
    \end{equation}
    and
    \begin{align}
        &\frac{x^c}{T}N^{\varepsilon-c}\frac{2N+|x-N|}{|x-N|}x^{-c+\kappa_3}(\log{x})^{-k+\kappa_2} \notag\\
       &\quad = \frac{1}{\kappa_1(\log{x})^k}\left(x-\frac{1}{2}\right)^{\varepsilon-c}\frac{2\left(x-\frac{1}{2}\right)+\left|x-\left(x-\frac{1}{2}\right)\right|}{\left|x-\left(x-\frac{1}{2}\right)\right|} \notag \\
       &\quad < \frac{4x}{\kappa_1(\log{x})^k}\cdot \left(x-\frac{1}{2}\right)^{\varepsilon-c} \label{eq:maxn0}
    \end{align}
     since $x$ is a half-integer. The right-hand sides of \eqref{eq:r1inr10} and \eqref{eq:maxn0} are decreasing functions of $x$ for all $x>1$. Hence, we can set $x=x_1$ in them. Moreover, we have
    \begin{equation*}
        \frac{x^{c}\left(2x-\frac{1}{2}\right)r_2(x,a)}{T x^{c-\kappa_3} (\log{x})^{k-\kappa_2}}\left(\frac{x}{a}\right)^{\varepsilon-c}<\frac{2x^{1+\varepsilon}r_2(x,a)}{\kappa_1 x^c(\log{x})^k}a^{c-\varepsilon},
    \end{equation*}
    and the right-hand side is decreasing with respect to $x \geq e^e+1/2$ since $1+\varepsilon<c$, $k \geq 4$ and $a \geq 1.6$. A similar idea applies to the terms coming from $r_3$ in \eqref{eq:sumdKNonInteger}. Thus, we can conclude that the right-hand side of \eqref{eq:sumdKNonInteger} multiplied by $(1+\pi)x^c/(\pi Tx^{c-\kappa_3} (\log{x})^{k-\kappa_2})$ is at most
    \begin{equation}
    \label{eq:r10}
        \leq \frac{r_9(x_1, \kappa_1(\log{x_1})^{\kappa_2}x_1^{\kappa_3}, k, \varepsilon, \varepsilon_1,a)}{ x_1^{c-\kappa_3} (\log{x_1})^{k-\kappa_2}}
    \end{equation}
    for $x \geq x_1$. This concludes the case of Lemma \ref{lemma:dkGeneral}.

    Let us now consider the terms coming from Lemma \ref{lemma:sigmaChanges}. Since $\varepsilon \leq 0.2$ and $k\geq 4$, we can estimate 
    \begin{multline}
        \label{eq:kappa1lower}
        \kappa_1=\left(0.611^4\cdot\frac{5\pi^2 (k+2)}{6}\right)^{\frac{6}{k+2}}\left(\frac{a_1(c,c,3000)}{0.611}\right)^{\frac{6k}{k+2}} \\
        >1\cdot\left(\frac{\zeta(1.2)}{0.611}\left(1+\frac{\pi}{2\log{3000}}\right)\right)^{\frac{6k}{k+2}} \geq \left(\frac{\zeta(1.2)}{0.611}\left(1+\frac{\pi}{2\log{3000}}\right)\right)^{4}>2.
    \end{multline}
    We can also verify that $\kappa_1>2$ if $(k,\varepsilon_1) \in \{(7,2/7), (8,7/32)\}$.
    It follows that
    \begin{multline}
    \label{eq:r11}
        \frac{r_5(x,1/2,c, \kappa_1(\log{x})^{\kappa_2}x^{\kappa_3},3000,\kappa_1(\log{x})^{\kappa_2}x^{\kappa_3}, k)}{x^{c-\kappa_3}(\log{x})^{k-\kappa_2}} \\
        \leq \frac{r_{10}(x_1,\kappa_1(\log{x_1})^{\kappa_2}x_1^{\kappa_3}, 3000,  \kappa_1x_1^{\kappa_3}, k, c)}{x_1^{c-\kappa_3}(\log{x_1})^{k-\kappa_2}}
    \end{multline}
    for all $x \geq x_1$.

    We consider the terms coming from Lemma \ref{lemma:fourthmomentExt} similarly. First, we note that for $T \geq 3000$ and $k \geq 4$, we have
        $l_k-11266.536\left(0.611T^{1/6}\log{T}\right)^{k-4}<0$,
    where $l_k$ is as in \eqref{def:l}. Combined with \eqref{eq:kappa1lower}, it follows that for all $x\geq x_1$ we have
    \begin{equation}
    \label{eq:SkNew}
        \frac{x^{1/2}\left(S_{k}(T)+l_{k}\right)}{x^{  \kappa_3\left(\frac{k-4}{6}\right)+\frac{1}{2}}(\log{x})^{  \kappa_2\left(\frac{k-4}{6}\right)+k+1}} \leq \frac{S_{k,\text{new}}(\kappa_1x_1^{\kappa_3})}{x_1^{  \kappa_3\left(\frac{k-4}{6}\right)}(\log{x_1})^{  \kappa_2\left(\frac{k-4}{6}\right)+k+1}},
    \end{equation}
 where $S_{k,\text{new}}$ is as in \eqref{def:SkNew}. Combining estimates \eqref{eq:r10}, \eqref{eq:r11} and \eqref{eq:SkNew}, we obtain that the constant term is $r_{11}$ if $x$ is a half-integer.
    
      Let us now consider the case where $x$ is an integer. We define $c$ and $T_0$ similarly as in the case of a half-integer. In this case, in Lemma \ref{lemma:TkIntegral}, we apply estimate \eqref{eq:integerdSumFirst} instead of estimate \eqref{eq:TkSum} and in Lemma \ref{lemma:dkGeneral}, we apply estimate \eqref{eq:sumDkInteger} instead of estimate \eqref{eq:sumdKNonInteger}. By Definition \ref{def:lambdak} and the fact $T \geq T_0$, we have
    \begin{equation}
    \label{eq:dkIntegerCase}
        \frac{d_k(x)}{2\pi}\left(\pi+\frac{2Tc}{T^2-c^2}\right) \leq \frac{x^{\frac{\lambda(k)}{\log\log{x}}}}{2\pi}\left(\pi+\frac{2T_0c}{T_0^2-c^2}\right). 
    \end{equation}
    The same choices for $\kappa_1$, $\kappa_2$ and $\kappa_3$ work in this case, too. The right-hand side of \eqref{eq:dkIntegerCase} divided by $x^{\frac{k-1+\varepsilon_1(k-4)}{k+2}}(\log{x})^{\frac{k^2+2k+6}{k+2}}$ is decreasing for all $x\geq x_1$ because of assumption \eqref{eq:x0lambda}, and hence we choose $x=x_1$ to obtain an upper bound. The constant term follows similarly as in the previous case but now we replace $r_9$ with $r_{12}$ which leads to the constant $r_{13}$ at the end. 
\end{proof}

In the proof of Theorem \ref{thm:main}, we would like to consider cases where $\varepsilon_1>0.2$, too. For those cases, we present the following remark.

\begin{remark}
\label{rmk:FourthFirstSmallk}
If 
\begin{equation*}
    (k,\varepsilon_1)\in
\left\{
\begin{aligned}
&(5,1.5), (5,0.8), (5,0.448), (6,0.8), (6,0.429),  (6,\tfrac{5}{12}),    \\
&(7,0.5), (8,0.4),(9,0.35)
\end{aligned}
\right\},
\end{equation*}
then $0<\kappa_1(\log{x})^{\kappa_2}<1$ for all $x\geq x_1$, where $x_1$ is as in Theorem \ref{thm:EffectiveFourth}, $\kappa_1$ as in \eqref{def:kappa1FourthFirst} and $\kappa_2$ as in \eqref{def:kappa2Fourth}. In these cases, we can bound $\Delta_k(x)$ otherwise similarly as in Theorem \ref{thm:EffectiveFourth} but for $T_{1}$ we will use $ x^{\kappa_3}$ instead of $\kappa_1x^{\kappa_3}$. 
\end{remark}

Next, we provide a sharper estimate for $\Delta_{k}(x)$ for all $k\ge 4$ when $x$ is large enough. Here, we define
\begin{align}
     &H_{k,\text{new}}(T):=(0.611T^{1/6}\log T)^{k-4}\Bigg(\frac{1}{10\pi^{3}}\log^{5}(T/2) +3.177\log^{4}(T/2)\notag\\
    &\quad+12.644\log^{3}(T/2)\Bigg)+\frac{205.760^{k}\cdot1.910}{k}-445579.419.\label{def:HkNew}
\end{align}

\begin{theorem}
    \label{thm:Tlarge}
    With the same assumptions and notation as in  Theorem~\ref{thm:EffectiveFourth}, we obtain
   \begin{align*}
        &\left|\Delta_k(x)\right|< \omega_2(x_1,k, \varepsilon, \varepsilon_1,a)x^{\frac{k-1+\varepsilon_1(k-4)}{k+2}}(\log{x})^{\frac{k^2+2k+6}{k+2}}, 
    \end{align*}
     where $\omega_2$ and $x_1$ are otherwise the same as $\omega_1$ and $x_1$ in Theorem \ref{thm:EffectiveFourth} but $S_{k,\text{new}}$ is replaced by $H_{k,\text{new}}$, $3000$ with $5.5\cdot10^7$ and 
    \begin{equation*}
\frac{5\pi^2(k+2)a_1(1+\varepsilon_1,1+\varepsilon_1,3000)^k}{6\cdot0.611^{k-4}}\text{ with } \frac{5\pi^2(k+2)a_1(1+\varepsilon_1,1+\varepsilon_1,5.5\cdot 10^{7})^k}{3\cdot0.611^{k-4}}.
    \end{equation*} 
\end{theorem}

\begin{proof}
    We proceed by following the same argument as in the proof of Theorem~\ref{thm:EffectiveFourth}. All the steps and identities used there remain valid in our present setting, except for the application of Lemma~\ref{lemma:fourthmomentExt}. Since we now consider the case $T\ge 5.5\cdot 10^{7}$ (instead of $T\ge 3000$), we replace Lemma~\ref{lemma:fourthmomentExt} with Lemma~\ref{lemma:fourthmomentEst}, which is tailored to this range. With this adjustment, a few of the resulting estimates differ slightly from those obtained in the previous proof. In particular, \eqref{eq:effectiveFourthHalfInteger} from the previous argument is replaced by
\begin{equation*}
        \frac{x^{\frac{1}{2}}}{10\pi^{3}}(0.611T^{1/6}\log T)^{k-4}\log^{5}(T/2),
    \end{equation*}
    and the constant $\kappa_{1}$ accordingly becomes
\begin{equation}\label{eq:kappa12}
\left(\frac{5\pi^2(k+2)a_1(c,c,5.5\cdot 10^{7})^k}{3\cdot0.611^{k-4}}\right)^{\frac{6}{k+2}}.
    \end{equation}
    
    Also, for $\varepsilon\le 0.2$ and $k\ge 4$, we have
    \begin{align}       
\kappa_1&=\left(0.611^4\cdot\frac{5\pi^2 (k+2)}{3}\right)^{\frac{6}{k+2}}\left(\frac{a_1(c,c,5.5\cdot 10^{7})}{0.611}\right)^{\frac{6k}{k+2}} \notag \\
        &>1\cdot\left(\frac{\zeta(1.2)}{0.611}\left(1+\frac{\pi}{2\log{(5.5\cdot 10^{7})}}\right)\right)^{\frac{6k}{k+2}}  \notag\\
        &\geq \left(\frac{\zeta(1.2)}{0.611}\left(1+\frac{\pi}{2\log{(5.5\cdot 10^{7})}}\right)\right)^{4}>2,\label{eq:kappa1lower2}
    \end{align}
    and hence
    \begin{multline}
    \label{eq:r11ct}
        \frac{r_5(x,1/2,c, \kappa_1(\log{x})^{\kappa_2}x^{\kappa_3},5.5\cdot10^{7},\kappa_1(\log{x})^{\kappa_2}x^{\kappa_3}, k)}{x^{c-\kappa_3}(\log{x})^{k-\kappa_2}} \\
        \leq \frac{r_{10}(x_1,\kappa_1(\log{x_1})^{\kappa_2}x_1^{\kappa_3}, 5.5\cdot 10^{7}, \kappa_1x_1^{\kappa_3}, k, c)}{x_1^{c-\kappa_3}(\log{x_1})^{k-\kappa_2}}
    \end{multline}
    for all $x\ge x_{1}$.

Next, we examine the contributions arising from Lemma~\ref{lemma:fourthmomentEst}. From \eqref{ckconstant}, since $205.760$ dominates the other bases for large $k$, we compress the $k$-dependence into a single exponential. Hence, for all $k\ge 4$ and $T\ge 5.5\cdot 10^{7}$, we have
\begin{align*}
c_{k}-243234.568(0.611T^{1/6}\log{T})^{k-4}&\le\frac{205.760^{k}\cdot1.910}{k}-445579.419.
\end{align*}
Since the right-hand side is positive for all $k\ge4$, combining this with \eqref{eq:kappa1lower2} yields
 \begin{equation}
    \label{eq:HkNew}
        \frac{x^{1/2}\left(H_{k}(T)+c_{k}\right)}{x^{  \kappa_3\left(\frac{k-4}{6}\right)+\frac{1}{2}}(\log{x})^{  \kappa_2\left(\frac{k-4}{6}\right)+k+1}} \leq \frac{H_{k,\text{new}}(\kappa_1x_1^{\kappa_3})}{x_1^{  \kappa_3\left(\frac{k-4}{6}\right)}(\log{x_1})^{  \kappa_2\left(\frac{k-4}{6}\right)+k+1}}
    \end{equation}
    for all $x\ge x_{1}$, where $H_{k,\text{new}}(T)$ is as defined in \ref{def:HkNew}. From the estimates in \eqref{eq:r10}, \eqref{eq:r11ct}, and \eqref{eq:HkNew}, it follows  that whenever $x$ is a half-integer, the constant term is $r_{11}$ with $S_{k,\text{new}}(T)$ replaced with $H_{k,\text{new}}(T)$, where we use our new $\kappa_1$ and take $T_{0}=5.5\cdot 10^{7}$. 

    For the case where $x$ an integer, the result is similar to that in the proof of Theorem~\ref{thm:EffectiveFourth}, with $T_{0}=5.5\cdot 10^{7}$.  
     The same choices of $\kappa_{1}, \kappa_{2}$, and $\kappa_{3}$ work here as well. 
As before, the right-hand side of~\eqref{eq:dkIntegerCase}, divided by $x^{\frac{k-1+\varepsilon_{1}(k-4)}{k+2}} (\log x)^{\frac{k^{2}+2k+6}{k+2}}$,
is decreasing for all $x \ge x_{1}$ by assumption \eqref{eq:x0lambda}, so we take $x = x_{1}$ to obtain an upper bound. The constant term is obtained similarly 
to the previous case, except that we now replace $S_{k, \text{new}}$ by $H_{k, \text{new}}$ in $r_{11}$, 
which yields the constant $r_{13}$.
\end{proof}

Again, we mention a few special cases that are used in the proof of Theorem \ref{thm:main}.
\begin{remark}
\label{rmk:FourthSecondSmallk}
If 
\begin{equation*}
    (k,\varepsilon_1)\in
\left\{
\begin{aligned}
&(4,0.361), (5,1.5), (5,0.8) ,(5,0.448), (6,0.8), (6,0.429), (6,\tfrac{5}{12}),\\
&(7,0.5),(8,0.4),(9,0.35),
\end{aligned}
\right\}
\end{equation*}
then $0<\kappa_1(\log{x})^{\kappa_2}<1$ for all $x \geq x_1$, where $x_1$ is as in Theorem \ref{thm:Tlarge}, $\kappa_1$ as in \eqref{eq:kappa12} and $\kappa_2$ as in \eqref{def:kappa2Fourth}. In these cases, we can bound $\Delta_k(x)$ otherwise similarly as in Theorem \ref{thm:Tlarge} but for $T_{1}$ we will use $x^{\kappa_3}$ instead of $\kappa_1x^{\kappa_3}$.
\end{remark}

\subsection{Effective results based on the second moment}
\label{sec:effectiveSecondMoment}

Next, we derive similar estimates using the second moment. Here, we define
\begin{align*}
   &  u_{k, \text{new}}:=u_k-0.458\cdot1.067^{k-2}, \\
   & V_{k,\text{new}}(T):=(0.611\,T^{\frac{1}{6}}\log T)^{k-2}\left(\frac{\log^{2}T}{2\pi}+0.425(\log{T})^{\frac{3}{2}}+7.658\log T \right. \\
   &\quad\left.+0.637\sqrt{\log T}\right)+u_{k, \text{new}}, \\
   & r_{14}(f,g,x, T, T_0, T_1, k, \varepsilon, \varepsilon_1,a):=\frac{r_9(x, T, k, \varepsilon, \varepsilon_1,a)}{x^{\frac{k+1+(k-2)\varepsilon_1}{k+4}}(\log{x})^k} \\
   &\quad\quad +\frac{g(x,T,T_0,T_1,k, 1+\varepsilon_1)+x^{\frac{1}{2}}f(T_1)}{x^{\frac{k+1+(k-2)\varepsilon_1}{k+4}}(\log{x})^k} \\
   &\text{and} \\
   & r_{15}(f,g, x, T, T_0, T_1, k, \varepsilon, \varepsilon_1,a):=\frac{r_{12}(x, T, T_0, k, \varepsilon, 1+\varepsilon_1,a)}{x^{\frac{k+1+(k-2)\varepsilon_1}{k+4}}(\log{x})^k} \\
   &\quad\quad +\frac{g(x,T,T_0,T_1,k, 1+\varepsilon_1)+x^{\frac{1}{2}}f(T_1)}{x^{\frac{k+1+(k-2)\varepsilon_1}{k+4}}(\log{x})^k},
\end{align*}
where $u_k$ is as in \eqref{def:uk}, and $r_9$ and $r_{12}$ as in \eqref{def:r10} and \eqref{def:r13} respectively.

\begin{theorem}
\label{thm:EffectiveSecond}
    Assume $k \geq 2$. Let $\lambda(k)$ be as in Definition \ref{def:lambdak}, $a_1$ as in \eqref{def:a1}, $\varepsilon$, $x_0$ and $h_k$ be as in Lemma \ref{lemma:dkGeneral}, $a \in[1.6, e^{h_k}]$ and $r_{10}$ as in \eqref{def:newr10}. Let us also suppose that $\varepsilon_1 \in(\varepsilon, \lambda(k)/\log\log{3}-1)$ and $\varepsilon_1\leq 1.3$. 
    
    If $x$ is a half-integer or an integer and $x\geq x_1$, then we have
   \begin{align*}
        &\left|\Delta_k(x)\right|< \omega_3(x_1,k, \varepsilon, \varepsilon_1,a)x^{\frac{k+1+\varepsilon_1(k-2)}{k+4      }}(\log{x})^{k}, 
    \end{align*}
    where 
    \begin{align*}
        \omega_3(x_1,k, \varepsilon, \varepsilon_1,a)&:=\max\left\{r_{14}(f,g,x_1, T, T_0,  T_1, k, \varepsilon, \varepsilon_1,a), \right. \\
       &\left.  r_{15}(f, g, x_1, T, T_0, T_1, k, \varepsilon, \varepsilon_1,a)  \right\},
    \end{align*}  
\begin{align*}
    &f=V_{k,\text{new}}, \text{ } \kappa_1=\left(0.611^2(1+2\varepsilon_1)\left(\frac{a_1(1+\varepsilon_1,1+\varepsilon_1,4)}{0.611}\right)^{k}\right)^{\frac{6}{k+4}}, \\
    &\kappa_3=\frac{3(1+2\varepsilon_1)}{k+4}, \text{ }T=T_1=\kappa_1x_1^{\kappa_3}, \text{ } T_0=4, \text{ } g=r_{10} 
\end{align*}
   and $x_1\geq \max\{ax_0,e^e+1/2\}$, $x_1>4a$ if 
   \begin{equation*}
       \lambda(k) < \frac{4(k+1+\varepsilon_1(k-2))}{k+4},
   \end{equation*}
   and
    \begin{align*}
         x_1\geq  & \max\left\{ax_0, e^e+\frac{1}{2}, \vphantom{\exp\left(\exp\left((k+4)\frac{\lambda(k)+\left(\lambda(k)^2-\lambda(k)\frac{4((k+1)+\varepsilon_1(k-2))}{k+4}\right)^{\frac{1}{2}}}{2(k+1+\varepsilon_1(k-2))}\right) \right)}\right. \\
        &\quad \left. \exp\left(\exp\left((k+4)\frac{\lambda(k)+\sqrt{\lambda(k)^2-\lambda(k)\frac{4((k+1)+\varepsilon_1(k-2))}{k+4}}}{2(k+1+\varepsilon_1(k-2))}\right) \right)\right\}, 
    \end{align*}
    $x_1>4a$ otherwise.
\end{theorem}

\begin{proof}
    The proof follows similarly as in Theorem \ref{thm:EffectiveFourth}. We will start with the case where $x$ is a half-integer. In this case, the estimates \eqref{eq:fourthNonIntegerFirst} and \eqref{eq:fourthNonIntegerSecond} remain unchanged, while we replace estimate \eqref{eq:effectiveFourthHalfInteger}  with
    \begin{equation}
    \label{eq:SecondMomentMain}
        \frac{0.611^{k-2}}{2\pi}x^{\frac{1}{2}}(\log{T})^kT^{\frac{k-2}{6}}
    \end{equation}
     by Lemma \ref{lemmaMM2}. Now we choose $c=1+\varepsilon_1$ and $T_0=4$.  

    As before, we would like to choose $\log{T_1}=\log{T}+O_k(\log\log{x})$ and
    $T=\kappa_1(\log{x})^{\kappa_2}x^{\kappa_3}$,
    and optimize $\kappa_1, \kappa_2$ and $\kappa_3$. Asymptotically, the largest terms come now from the the last term in \eqref{eq:fourthNonIntegerSecond} and \eqref{eq:SecondMomentMain}. To optimize the terms $\kappa_i$ for large $x$, we want to have
    \begin{equation*}
        \kappa_3\left(\frac{k-2}{6}\right)+\frac{1}{2}=1+\varepsilon_1-\kappa_3 \quad\Leftrightarrow \quad \kappa_3=\frac{3(1+2\varepsilon_1)}{k+4},
    \end{equation*}
    \begin{equation*}
        \kappa_2\left(\frac{k-2}{6}\right)+k=k-\kappa_2 \quad\Leftrightarrow \quad \kappa_2=0
    \end{equation*}
    and
    \begin{multline*}
        \frac{0.611^{k-2}}{2\pi}\kappa_3^{k}\kappa_1^{\frac{k-2}{6}}=a_1(c,c,4)^k\frac{\kappa_3^k}{\kappa_1 \pi}\left(\frac{1}{2}+\varepsilon_1\right) \\
        \Leftrightarrow \quad \kappa_1=\left(0.611^2(1+2\varepsilon_1)\left(\frac{a_1(c,c,4)}{0.611}\right)^{k}\right)^{\frac{6}{k+4}}.
    \end{multline*}
    Now,
    \begin{equation*}
       1+\varepsilon_1-\kappa_3= \frac{k+1+(k-2)\varepsilon_1}{k+4}
    \end{equation*} 
    gives the exponent of $x$ and $k-\kappa_2=k$ the exponent of $\log{x}$ in our bound for $|\Delta_k(x)|$.

    Next, we concentrate on deriving the constant in front of estimate of $|\Delta_k(x)|$ applying the previous values $\kappa_i$. First, we notice that the negative part in \eqref{def:Vk} can be estimated as
    \begin{equation*}
        -0.458\left(0.611 T^{\frac{1}{6}}\log{T}\right)^{k-2} \leq -0.458\left(0.611\cdot 4^{\frac{1}{6}}\log{4}\right)^{k-2}\leq -0.458\cdot1.067^{k-2}
    \end{equation*}
    for $T\geq T_0=4$. Since $u_{k, \text{new}}=u_k-0.458\cdot1.067^{k-2}>0$ for all $k \geq 2$ and 
    \begin{equation*}
        \kappa_1>0.611^{\frac{12-k}{k+4}}\left(\frac{\zeta(c)}{\log{4}}\right)^{\frac{6k}{k+4}}\geq 0.611^{\frac{5}{3}}\left(\frac{\zeta(2.3)}{\log{4}}\left(1+\frac{\pi}{2\log{4}}\right)\right)^2>1
    \end{equation*} 
    due to the assumption $\varepsilon_1 \leq 1.3$.  Hence, we choose $T_1=T$ and we have
    \begin{equation*}
       \frac{x^{\frac{1}{2}}\left(V_{k}\left(\kappa_1x^{\kappa_3}\right)+u_{k}\right)}{x^{1+\varepsilon_1-\kappa_3}(\log{x})^k}\leq  \frac{x^{\frac{1}{2}}V_{k,\text{new}}\left(\kappa_1x^{\kappa_3}\right)}{x^{1+\varepsilon_1-\kappa_3}(\log{x})^k},
    \end{equation*}
    and the right-hand side is decreasing for $x$. Furthermore, once the right-hand side of \eqref{eq:maxn0} is normalized by $x^{1+\varepsilon_1-\kappa_3}(\log{x})^k$, the result   decreases with $x$. The term
    \begin{equation*}
        \frac{x^{c} }{T x^{c-\kappa_3}(\log{x})^k}\left(x+\frac{1}{2}\right)^{\varepsilon-\varepsilon_1}r_3(x,a)
    \end{equation*}
    is also a decreasing function of $x$ since $k\geq 2$ and $x>1$. The same conclusion applies for other terms coming from $r_9$. Thus,
    $r_{14}(f,x, T, T_0, T_1, k, \varepsilon, \varepsilon_1, a)
    $
    is decreasing for all $x \geq x_1$, and produces the desired constant when we set $x=x_1$.

    Let us now consider the case where $x$ is integer. The approach is also the same as in the case of a half-integer except we apply estimate \eqref{eq:integerdSumFirst} instead of estimate \eqref{eq:TkSum} (and hence estimate \eqref{eq:dkIntegerCase}), estimate \eqref{eq:sumDkInteger} instead of estimate \eqref{eq:sumdKNonInteger}. These changes do not affect on our choices for $\kappa_1, \kappa_2$ and $\kappa_3$, and hence also $x_1$ stays the same. It is worth noting that the right-hand side of \eqref{eq:dkIntegerCase} divided by $x^{1+\varepsilon_1-\kappa_3}(\log{x})^k$ is decreasing for all $x \geq x_1$ because of the definition of $x_1$. Using similar reasoning as in the previous paragraph, we can conclude that if $x$ is an integer, we can replace $r_{14}$ with $r_{15}$.

    Lastly, since $T_0 \geq 4$ in Lemma \ref{lemmaMM2}, we must have
    \begin{equation*}
       x_1 \geq \frac{4^{\frac{k+4}{3(1+2\varepsilon_1)}}\cdot0.611^{\frac{2k}{1+2\varepsilon_1}}}{\left(0.611^2(1+2\varepsilon_1)a_1(c,c,4)^k\right)^{\frac{2}{1+\varepsilon_1}}}, 
    \end{equation*}
    and by Lemma \ref{lemma:dkGeneral}, we require that $x_1 \geq e^e+1/2$. However, since $\varepsilon_1 \leq 1.3$ we have
    \begin{equation*}
        \frac{4^{\frac{k+4}{3(1+2\varepsilon_1)}}\cdot0.611^{\frac{2k}{1+2\varepsilon_1}}}{\left(0.611^2(1+2\varepsilon_1)a_1(c,c,4)^k\right)^{\frac{2}{1+\varepsilon_1}}} <4^{\frac{k+4}{3}} \cdot 0.611^{2k}<e^e+\frac{1}{2}
    \end{equation*}
    for all $k\geq 2$. Hence, we can conclude
    \begin{equation*}
        \max\left\{e^e+\frac{1}{2}, \frac{4^{\frac{k+4}{3(1+2\varepsilon_1)}}\cdot0.611^{\frac{2k}{1+2\varepsilon_1}}}{\left(0.611^2(1+2\varepsilon_1)a_1(c,c,4)^k\right)^{\frac{2}{1+\varepsilon_1}}} \right\}=e^e+\frac{1}{2},
    \end{equation*}
    and use that in the lower bound for $x_1$.
\end{proof}

As before, we provide a different version of the result if $x$ is large enough. Prior to this, we define
\begin{equation*}
    b_{k,\text{new}}:=b_{k}-756.444(3.977\cdot10^6)^{k-2}
\end{equation*}
and
\begin{align*}
    G_{k,\text{new}}(T):&=(0.566\,T^{\frac{1}{6}}\log T)^{k-2}\left(\frac{\log^{2} T}{2\pi}+\frac{3v(\log{T})^{5/3}}{2(1.1\cdot 10^{30})^{2/3}\pi}\right)+b_{k,\text{new}},
\end{align*}
where  $\upsilon := 1271506.721$ and $b_k$ as in \eqref{def:bk}. 
Due to our choices of $k$ and $\varepsilon_1$ described in Section \ref{sec:ProofMain}, we want to consider two cases: $\kappa_1 \in (0,1)$ and $\kappa_1 \geq 1$. Hence, we present the results in those two cases in the next theorem.

\begin{theorem}
\label{thm:EffectiveSecond2}
Let $k, \lambda(k), \varepsilon, \varepsilon_1, h_k, x_0$ and $x_1$ be as in Theorem \ref{thm:EffectiveSecond}, $a_3$ as in Remark \ref{rmk:a3} and $r_6$ as in \eqref{def:newr6}. If $x$ is a half-integer or an integer and $x\geq x_1$, then we have
 \begin{align*}
        &\left|\Delta_k(x)\right|< \omega_4(x_1,k, \varepsilon, \varepsilon_1,a)x^{\frac{k+1+(k-2)\varepsilon_1}{k+4}}(\log{x})^k, 
    \end{align*}
    where
    \begin{align*}
        \omega_4(x_1,k, \varepsilon, \varepsilon_1,a)&:=\max\left\{r_{14}(f,g,x_1, T, T_0,  T_1, k, \varepsilon, \varepsilon_1,a), \right. \\
       &\left.  r_{15}(f, g, x_1, T, T_0, T_1, k, \varepsilon, \varepsilon_1,a)  \right\},
    \end{align*}  
\begin{align}
    &f=G_{k,\text{new}}, \text{ } \kappa_1=\left(\frac{(1+2\varepsilon_{1})a_{3}(1+\varepsilon_1,1+\varepsilon_1,1.1\cdot 10^{30})^k}{0.566^{k-2}}\right)^{\frac{6}{k+4}}, \label{eq:kappa1secondsecond} \\
    &\kappa_3=\frac{3(1+2\varepsilon_1)}{k+4}, \text{ }T=\kappa_1x_1^{\kappa_3}, \text{ } T_0=1.1\cdot10^{30}, \text{ } g=r_{6} \notag
\end{align}
and
\begin{equation*}
    T_1=
    \begin{cases}
    \kappa_1x_1^{\kappa_3}, &\text{if } \kappa_1\geq 1 \\
    x_1^{\kappa_3}, &\text{otherwise}.
    \end{cases}
\end{equation*}
In addition, we also assume that $x_1 \geq \left(1.1\cdot 10^{30}\right)^{\frac{k+4}{3(1+\varepsilon_1)}}\kappa_1^{-\frac{2}{1+\varepsilon}}$.
\end{theorem}
    
    \begin{proof}
  The proof follows the same strategy as in Theorem~\ref{thm:EffectiveSecond}, and all intermediate steps remain unchanged, except that a different estimate is required in place of Lemma~\ref{lemmaMM2}. Since we consider $T \ge 1.1\cdot 10^{30}$ rather than $T \ge 4$, we use  Lemma~\ref{lemmaMM21} in place of Lemma~\ref{lemmaMM2} and apply Remark \ref{rmk:newa6}. The choices of $c$ and $T$, as well as the values of $\kappa_{2}$ and $\kappa_{3}$ remain the same. However, we have $\kappa_1$ as in \eqref{eq:kappa1secondsecond}.

  For $T\ge T_{0}=1.1\cdot 10^{30}$, we estimate the negative part of   \eqref{nega39} as follows, 
    \begin{align*}
       -(0.566T^{1/6}\log{T})^{k-2}&\Bigg(0.217\log{T}\\
       &\quad+\frac{v\log^{5/3}T}{2\pi T^{2/3}}+741.434\Bigg)\le -756.444(3.977\cdot10^6)^{k-2}.
   \end{align*}
   Consequently, for all $k\ge 2$, we have $b_{k,\text{new}}=b_{k}-756.444(3.977\cdot10^6)^{k-2}> 0.$
      In Theorem~\ref{thm:EffectiveSecond}, we can omit the lower bound for $x$ coming from Lemma~\ref{lemmaMM2}, since the lower bound coming from Lemma \ref{lemma:dkGeneral} is always larger. However, this is not the case anymore with Lemma \ref{lemmaMM21}, and hence we must have
    \begin{equation*}
       x_{1} \geq \frac{(1.1\cdot 10^{30})^{\frac{k+4}{3(1+2\varepsilon_1)}}\cdot0.566^{\frac{2k}{1+2\varepsilon_1}}}{\left(0.566^2(1+2\varepsilon_1)a_3(c,c,1.1\cdot 10^{30})^k\right)^{\frac{2}{1+\varepsilon_1}}}. 
    \end{equation*}
    Next, we want to consider two cases $\kappa_1 \in (0,1)$ and $\kappa_1 \geq 1$.
    
The rest of the case $\kappa_1\geq 1$ is similar to the proof of Theorem~\ref{thm:EffectiveSecond} with the replacements mentioned in the statement of the theorem. 
In the case $0<\kappa_{1}<1$, we have $\kappa_{1}x^{\kappa_{3}}<x^{\kappa_{3}}$ for all $x>0$. 
Hence, we obtain
\begin{equation*}
\frac{x^{\frac{1}{2}}\left(G_{k}\left(\kappa_1x^{\kappa_3}\right)+b_{k}\right)}{x^{1+\varepsilon_1-\kappa_3}(\log{x})^k}< \frac{x^{\frac{1}{2}}G_{k,\text{new}}\left(x^{\kappa_3}\right)}{x^{1+\varepsilon_1-\kappa_3}(\log{x})^k}
    \end{equation*}
and the right-hand side decreases with $x$.

    Moreover, in the case where $x$ is an integer, the argument proceeds as above.
\end{proof}

\begin{remark}
    From the proofs of Theorems \ref{thm:EffectiveSecond} and \ref{thm:EffectiveSecond2}, constants $r_{14}$ and $r_{15}$ in $\omega_3$ and $\omega_{4}$ correspond to the cases $x$ is an half-integer and $x$ is an integer, respectively.  
\end{remark}

\subsection{An effective result based on the general estimate for an integral}

Lastly, we prove an estimate $\Delta_{k}(x)$ by applying Lemma \ref{lemma:integralwithoutMoments}. We define
\begin{align}
&   C_{k,2}(b):=\frac{8\zeta^{k}(1-b)}{\pi\sqrt{k}}\left(\frac{(1-b)^2+9}{36}\right)^{-\frac{kb}{4}}\left(\frac{b^2+9}{36}\right)^{\frac{k(1-b)}{4}}  \notag \\
&\quad\cdot\exp\left(k\left(\frac{18b^2-18b+19}{36}\right)\right)C_{k,1}(b), \label{def:Ck2} \\
  & r_{16}(x, T, k, \varepsilon, \varepsilon_1,a):=\frac{r_9(x, T, k, \varepsilon, \varepsilon_1,a)}{x^{\frac{(1+2\varepsilon)(k-1)+\varepsilon_1(k(1+2\varepsilon)-1)}{k(1+2\varepsilon)+1}}(\log{x})^{\frac{k(k(1+2\varepsilon)-1)}{k(1+2\varepsilon)+1}}} \notag\\
   &\quad\quad +\frac{r_5(x,-\varepsilon,1+\varepsilon_1,T,3,T,k)+2^k+x^{-\varepsilon}\left(B(-\varepsilon,T)+A(-\varepsilon,k)\right)}{x^{\frac{(1+2\varepsilon)(k-1)+\varepsilon_1(k(1+2\varepsilon)-1)}{k(1+2\varepsilon)+1}}(\log{x})^{\frac{k(k(1+2\varepsilon)-1)}{k(1+2\varepsilon)+1}}} \notag \\
   &\text{and} \notag \\
   & r_{17}(x, T, k, \varepsilon, \varepsilon_1,a):=\frac{r_{12}(x, T, T_0, k, \varepsilon, 1+\varepsilon_1,a)}{x^{\frac{(1+2\varepsilon)(k-1)+\varepsilon_1(k(1+2\varepsilon)-1)}{k(1+2\varepsilon)+1}}(\log{x})^{\frac{k(k(1+2\varepsilon)-1)}{k(1+2\varepsilon)+1}}} \notag \\
   &\quad\quad +\frac{r_5(x,-\varepsilon,1+\varepsilon_1,T,3,T,k)+2^k+x^{-\varepsilon}\left(B(-\varepsilon,T)+A(-\varepsilon,k)\right)}{x^{\frac{(1+2\varepsilon)(k-1)+\varepsilon_1(k(1+2\varepsilon)-1)}{k(1+2\varepsilon)+1}}(\log{x})^{\frac{k(k(1+2\varepsilon)-1)}{k(1+2\varepsilon)+1}}}, \notag
\end{align}
where $C_{k,1}(b)$ is as in \eqref{consk}, $r_5, r_9$ and $r_{12}$ as in \eqref{def:r7}, \eqref{def:r10} and \eqref{def:r13}, respectively, and $A$ and $B$ as in Lemma \ref{lemma:integralwithoutMoments}. Moreover, let $\text{LW}_{0}(x)$ denote the Lambert $W$ function with the principal branch.

\begin{theorem}
\label{thm:nomoments} 
    Assume $k\geq 2$. Let $\lambda(k)$ be as in Definition \ref{def:lambdak}, $a_1$ as in \eqref{def:a1}, $\varepsilon$, $h_k$ and $x_0$ as in Lemma \ref{lemma:dkGeneral}. Assume also that $ a\in [1.6,e^{h_k}]$ $\varepsilon<0.31$, $\varepsilon_1 \in(\varepsilon, \lambda(k)/\log\log{3}-1)$.
    
    If $x$ is a half-integer or an integer, then
    \begin{equation*}
        \left|\Delta_k(x)\right|< \omega_5(x_1,k, \varepsilon, \varepsilon_1,a)x^{\frac{(1+2\varepsilon)(k-1)+\varepsilon_1(k(1+2\varepsilon)-1)}{k(1+2\varepsilon)+1}}(\log{x})^{\frac{k(k(1+2\varepsilon)-1)}{k(1+2\varepsilon)+1}}.
    \end{equation*}
    Here we have 
     \begin{align*}
        \omega_5(x_1,k, \varepsilon, \varepsilon_1,a)&:=\max\left\{r_{16}(x_1, T, k, \varepsilon, \varepsilon_1,a), r_{17}(x_1, T, k, \varepsilon, \varepsilon_1,a)   \right\},
    \end{align*}  
\begin{align*}
    &\kappa_1=\left(\frac{(1/2+\varepsilon_1)(a_1(1+\varepsilon_1,1+\varepsilon_1,3)\kappa_3)^k}{C_{k,2}(-\varepsilon) \pi}\right)^{\frac{2}{k(1+2\varepsilon)+1}}, \text{ } \\
    & \kappa_2=\frac{2k}{k(1+2\varepsilon)+1}, \text{ }\kappa_3=\frac{2(1+\varepsilon_1+\varepsilon)}{k(1+2\varepsilon)+1}, \text{ }T=\kappa_1(\log{x_1})^{\kappa_2}x_1^{\kappa_3},
\end{align*}
      \begin{align*}
         &x_1\geq  \max\left\{ax_0, e^e+\frac{1}{2},  \exp\left(\frac{k}{1+\varepsilon_1+\varepsilon} \cdot\text{LW}_0\left(\frac{1+\varepsilon_1+\varepsilon}{k}\cdot \left(\frac{3}{\kappa_1}\right)^{\frac{1}{\kappa_2}}\right)\right), \right.\\
         &\left. \quad \exp\left(\kappa_1^{-\frac{1}{\kappa_2}}\right) \vphantom{\exp\left(\exp\left((k(1-2b)+1)\frac{\lambda(k)+\left(\lambda(k)^2-\lambda(k)\frac{4((1-2b)(k(1+\varepsilon_1)-1)-\varepsilon_1)}{k(1+2\varepsilon)+1}\right)^{\frac{1}{2}}}{2(k+1+\varepsilon_1(k-2))}\right) \right)}\right\}, 
    \end{align*}
    $x>4a$, if
    \begin{equation*}
       \lambda(k) < \frac{4((1+2\varepsilon)(k(1+\varepsilon_1)-1)-\varepsilon_1)}{k(1+2\varepsilon)+1},
   \end{equation*}
   and otherwise we assume also that 
   \begin{align*}
       &x_1\geq   \exp\left(\exp\left((k(1+2\varepsilon)+1)\frac{\lambda(k)+\sqrt{\lambda(k)^2-\lambda(k)\frac{4((1+2\varepsilon)(k(1+\varepsilon_1)-1)-\varepsilon_1)}{k(1+2\varepsilon)+1}}}{2((1+2\varepsilon)(k(1+\varepsilon_1)-1)-\varepsilon_1)}\right) \right).
    \end{align*}
\end{theorem}

\begin{proof}
In this case, we add the term $(-2)^{-k}\leq 2^{-k}$ to the right-hand side of \eqref{eq:zetakRectangle}. Then the first term in $r_5$ (see \eqref{def:r7}) will be positive, and we apply Lemma \ref{lemma:integralwithoutMoments} instead of using moment estimates. 

We can again set $T=\kappa_1(\log{x})^{\kappa_2}x^{\kappa_3}$,  $\log{T_1}=\log{T}+O_k(\log\log{x})$, $T_0=3$ and $c=1+\varepsilon_1$. Let also $b \in (-0.31,0)$. Hence, we want to have
\begin{equation*}
    \kappa_3k\left(\frac{1}{2}-b\right)-\frac{\kappa_3}{2}+b=
    \begin{cases}
    \kappa_3\left(\frac{k}{6}-1\right)+\frac{1}{2}, &\text{if } \kappa_3 \geq \frac{3(1+2\varepsilon_1)}{k} \\
    1+\varepsilon_1-\kappa_3, &\text{otherwise}
    \end{cases}
\end{equation*}
meaning that 
\begin{equation*}
    \kappa_3=\frac{2(1+\varepsilon_1-b)}{k(1-2b)+1}.
\end{equation*}
This, however, implies that the largest terms of $r_5$ is of size $O(T^{-1}(\log{T})^kx^c)$. Indeed, we also want to have
\begin{equation*}
    \kappa_2k\left(\frac{1}{2}-b\right)-\frac{\kappa_2}{2}=k-\kappa_2 \quad\iff\quad \kappa_2=\frac{2k}{k(1-2b)+1}
\end{equation*}
and
\begin{multline*}
    C_{k,2}(b)\kappa_1^{k(1/2-b)-1/2}=a_1(c,c,3)^k\frac{\kappa_3^k}{\kappa_1 \pi}\left(\frac{1}{2}+\varepsilon_1\right) \\
    \iff\quad \kappa_1=\left(\frac{(1/2+\varepsilon_1)(a_1(c,c,3)\kappa_3)^k}{C_{k,2}(b) \pi}\right)^{\frac{2}{k(1-2b)+1}},
\end{multline*}
where $C_{k,2}(b)$ is as in \eqref{def:Ck2}. Indeed, the exponent of $x$ in our estimate is
\begin{equation}
\label{eq:exponentxLastCase}
   1+\varepsilon_1-\kappa_3= \frac{(1-2b)(k-1)+\varepsilon_1(k(1-2b)-1)}{k(1-2b)+1},
\end{equation}
and the exponent of logarithm is
\begin{equation*}
    k-\kappa_2=\frac{k(k(1-2b)-1)}{k(1-2b)+1}.
\end{equation*}
We want the term $\varepsilon_1$ to be small in order to have small exponent in \eqref{eq:exponentxLastCase}. Since $\varepsilon<\varepsilon_1$, we will choose $b=-\varepsilon$ to get a small exponent \eqref{eq:exponentxLastCase} for $x$.

Now we proceed to study the constant in front of the estimate. Similar to the previous proofs, we can conclude that
\begin{equation*}
    \frac{r_9(\kappa_1(\log{x})^{\kappa_2}x^{\kappa_3},k,\varepsilon,\varepsilon_1,a)}{x^{c-\kappa_3}(\log{x})^{k-\kappa_2}} \quad\text{and}\quad \frac{r_{12}(\kappa_1(\log{x})^{\kappa_2}x^{\kappa_3},3,k,\varepsilon,c,a)}{x^{c-\kappa_3}(\log{x})^{k-\kappa_2}}
\end{equation*}
are decreasing for all $k \geq 2$, $a \geq 1.6$ and $x \geq x_0$. Since 
\begin{equation*}
    \log{x} \geq \log{x_0} \geq \frac{k(1+2\varepsilon)+1}{2a_1(c, c,3)(1+\varepsilon_1-\varepsilon)}\left(\frac{C_{k,2}(-\varepsilon) \pi}{\frac{1}{2}+\varepsilon_1}\right)^{\frac{1}{k}},
\end{equation*}
we have $\kappa_1(\log{x})^{\kappa_2}\geq 1$ and we choose $T_1=T$. Thus, the term
\begin{equation*}
    \frac{r_5(x,-\varepsilon,c,\kappa_1(\log{x})^{\kappa_2}x^{\kappa_3},3,\kappa_1(\log{x})^{\kappa_2}x^{\kappa_3},k)+2^k}{x^{c-\kappa_3}(\log{x})^{k-\kappa_2}}
\end{equation*}
is decreasing for all $x \geq x_0$. Due to our choices of $\kappa_i$, the term
\begin{equation*}
    \frac{x^{-\varepsilon}(B(-\varepsilon,\kappa_1(\log{x})^{\kappa_2}x^{\kappa_3})+A(-\varepsilon,k))}{x^{c-\kappa_3}(\log{x})^{k-\kappa_2}}
\end{equation*}
is also decreasing for all $x \geq x_0$. Here, $A$ and $B$ are as in Lemma \ref{lemma:integralwithoutMoments}. Hence, we obtain $r_{16}$ as a constant term if $x$ is a half-integer, and $r_{17}$ if $x$ is an integer.
\end{proof}

\begin{remark}
    Again, we see from the proof of Theorem \ref{thm:nomoments} that the constant $r_{16}$ corresponds to the case of half-integer and $r_{17}$ to the case where $x$ is an integer.
\end{remark}

\section{Proof of Theorem \ref{thm:main}}
\label{sec:ProofMain}

In this section, we prove Theorem \ref{thm:main}. It follows from effective results in Section \ref{sec:Effective}. We prove the result considering different cases depending on which theorem in Section \ref{sec:Effective} we use. All computations were carried out with \textit{Mathematica} version 12.0.0.0 and files are available upon request.

In all of the cases and for each $k$, we choose 
$x_0=\max\left\{4,\left\lceil\exp\left(\exp\left(\lambda(k)/\varepsilon\right)\right)\right\rceil \right\}$. Let us first consider the results using the fourth moment results given in Section \ref{sec:EffectiveFourth}. We keep in mind that we want to improve the exponent $(k-1)/k$ of $x$ in \cite[Theorem 1.2]{Tudzi2025} for all $k \geq 4$. Hence, for simplicity, we first choose $\varepsilon_1$ in such a way that
\begin{equation*}
    \frac{\varepsilon_1(k-4)}{k+2}=\frac{1}{2}\left(\frac{k-1}{k}-\frac{k-1}{k+2}\right) \quad\iff\quad \varepsilon_1=\frac{k-1}{k(k-4)}
\end{equation*}
and $\varepsilon=(k+2)/(k(k-4))$
for $k\geq 5$. Let us also set $a=1.6$ if $k \geq 5$. We consider only cases where we can choose $x_1$ be smaller than $10^{100}$ and hence we consider only cases $k=5$ and $k=6$ with these choices of $\varepsilon$ and $\varepsilon_1$. 

Next, we want to choose $\varepsilon$ and $\varepsilon_1$ in such a way that $x_1$ is as small as possible so that we can present reasonable bounds for larger values $k$ while still improving the known earlier estimates asymptotically. We notice that
\begin{equation*}
    \left(\frac{k-1}{k}-\frac{k-1}{k+2}\right)\frac{k+2}{k-4}
\end{equation*}
equals to $1.6$, $5/6\approx 0.833$, $4/7 \approx 0.571$, $7/16 \approx 0.438$ and $16/45\approx 0.356$ when $k=5, 6, 7, 8, 9$, respectively. We choose $\varepsilon_1=1.5$ if $k=5$, $\varepsilon_1=0.8$ if $k=6$, $\varepsilon_1=0.5$ if $k=7$, $\varepsilon_1=0.4$ if $k=8$ and $\varepsilon_1=0.35$ if $k=9$.  Starting from $k=10$, $\varepsilon_1$ can be at most $0.3$ forcing the lower bound of $x_1$ to be larger than $10^{100}$. Hence, values $k \geq 10$ are not included in these computations. Moreover, we recall that there are also other factors than $x_0$ that affect how small the term $x_1$ can be. Indeed, choosing $\varepsilon$ as large as possible (which allows $x_0$ to be small) does not necessarily mean a lower value for $x_1$. With $k=5,6$ it is beneficial to choose $\varepsilon$ as small as possible so that $x_1$ does not change. Thus, we choose $\varepsilon=0.863$ in Remark \ref{rmk:FourthFirstSmallk} and $\varepsilon=0.619$ in \ref{rmk:FourthSecondSmallk} if $k=5$, $\varepsilon=0.697$ in Remark \ref{rmk:FourthFirstSmallk} and $\varepsilon=0.524$ in \ref{rmk:FourthSecondSmallk} if $k=6$. For $(k,\varepsilon_1)=(7,0.5), (8, 0.4)$ and $(9,0.35)$, the lower bound for $x_1$ in Remarks \ref{rmk:FourthFirstSmallk} and \ref{rmk:FourthSecondSmallk} comes from the term $ax_0$, so we choose $\varepsilon=0.45$ at $k=7$, $\varepsilon=0.35$ at $k=8$ and $\varepsilon=0.349$ at $k=9$.

In \cite[Table 1]{Tudzi2025}, there are explicit results for the cases $k=5,6,$ and \cite[Corollary 1.3]{Tudzi2025} provides an estimate in the cases $k \geq 7$. We recognize that the results where we tried to find as small $x_1$ as possible for $k=7$ and $k=8$ already improve \cite[Corollary 1.3]{Tudzi2025} for all $x\geq x_1$. In the cases $k=5,6$, we want to find $\varepsilon, \varepsilon_1$ such that our results improve \cite[Table 1]{Tudzi2025} for as small $x$ as possible. For simplicity, let us choose $\varepsilon=\varepsilon_1-0.05$. The result \cite[Table 1]{Tudzi2025} for $k=5$ holds for $x\geq 10^{12}$. Hence, we try for each $x \approx 10^{12}, 10^{13}, \ldots$ which is the first one that would improve Tudzi's result for $k=5$. This is obtained at $(x_0,\varepsilon, \varepsilon_1)=(10^{29}, 0.398, 0.448)$. Similarly, for $k=6$, we choose $(x_0,\varepsilon,\varepsilon_1)=(10^{39}, 0.379, 0.429)$ and $(x_0,\varepsilon,\varepsilon_1)=(10^{72}, 0.348, 0.349)$ for $k=9$. All of the estimates using the fourth moment are described in Table \ref{fourthmomentk5678}.

\begin{center}
\begin{longtable}{ |c|c|c|c| } 
 \hline
 $(k,\varepsilon,\varepsilon_1,x_0)$ & Remark & $|\Delta_k(x)| <$... & $x_1$ \\ \hline
 $(5,0.6,0.8,\num{10969414})$ & \ref{rmk:FourthFirstSmallk} & $0.6449x^{\frac{24}{35}}(\log{x})^{\frac{41}{7}}$ & $1.7552\cdot10^7$ \\ \hline
  $(5,0.6,0.8,\num{10969414})$ &\ref{rmk:FourthSecondSmallk} & $24.1196x^{\frac{24}{35}}(\log{x})^{\frac{41}{7}}$ & $7.1314\cdot10^9$ \\ \hline
  $(5,0.863,1.5,1029)$ & \ref{rmk:FourthFirstSmallk} & $9.2710x^{\frac{11}{14}}(\log{x})^{\frac{41}{7}}$ & $1666$ \\ \hline
   $(5,0.619,1.5,\num{2905558})$ & \ref{rmk:FourthSecondSmallk} & $194.5861x^{\frac{11}{14}}(\log{x})^{\frac{41}{7}}$ & $4.8769\cdot 10^6$ \\ \hline
   $(5,0.398,0.448,10^{29})$ & \ref{rmk:FourthSecondSmallk} & $0.0207x^{\frac{556}{875}}(\log{x})^{\frac{41}{7}}$ & $1.6000\cdot 10^{29}$ \\ \hline
 $(6,\frac{1}{3},\frac{5}{12},7.1821\cdot10^{71})$ &  \ref{rmk:FourthFirstSmallk} & $0.0144x^{\frac{35}{48}}(\log{x})^{\frac{27}{4}}$ & $1.1492\cdot10^{72}$ \\ \hline
 $(6,\frac{1}{3},\frac{5}{12},7.1821\cdot10^{71})$ & \ref{rmk:FourthSecondSmallk} & $0.0031x^{\frac{35}{48}}(\log{x})^{\frac{27}{4}}$ & $1.1492\cdot10^{72}$ \\ \hline
 $(6,0.697,0.8,\num{99673})$ &  \ref{rmk:FourthFirstSmallk} & $0.2737x^{\frac{33}{40}}(\log{x})^{\frac{27}{4}}$ & $\num{162726}$ \\ \hline
 $(6,0.524,0.8,1.5794\cdot10^{11})$ & \ref{rmk:FourthSecondSmallk} & $1.3359x^{\frac{33}{40}}(\log{x})^{\frac{27}{4}}$ & $2.5590\cdot10^{11}$ \\ \hline
 $(6,0.379, 0.429,10^{39})$ & \ref{rmk:FourthSecondSmallk} & $0.0050x^{\frac{2929}{40000}}(\log{x})^{\frac{27}{4}}$ & $1.6000\cdot10^{39}$ \\ \hline
 $(7,0.45,0.5,1.9555\cdot10^{20})$ &  \ref{rmk:FourthFirstSmallk} & $0.0090x^{\frac{5}{6}}(\log{x})^{\frac{23}{3}}$ & $3.1288\cdot10^{20}$ \\ \hline
 $(7,0.45,0.5,1.9555\cdot10^{20})$ &  \ref{rmk:FourthSecondSmallk} & $0.0036x^{\frac{5}{6}}(\log{x})^{\frac{23}{3}}$ & $3.1288\cdot10^{20}$ \\ \hline
 $(8,0.35,0.4,2.3107\cdot10^{65})$ &  \ref{rmk:FourthFirstSmallk} & $0.0007x^{\frac{43}{50}}(\log{x})^{\frac{43}{5}}$ & $3.6972\cdot10^{65}$ \\ \hline
 $(8,0.35,0.4,2.3107\cdot10^{65})$ & \ref{rmk:FourthSecondSmallk} & $0.0002x^{\frac{43}{50}}(\log{x})^{\frac{43}{5}}$ & $3.6972\cdot10^{65}$ \\ \hline
 $(9,0.349,0.35,7.7210\cdot10^{70})$ & \ref{rmk:FourthFirstSmallk} & $7.0882\cdot10^{-5}x^{\frac{39}{44}}(\log{x})^{\frac{105}{11}}$ & $1.2354\cdot10^{71}$ \\ \hline
 $(9,0.349,0.35,7.7210\cdot10^{70})$ &  \ref{rmk:FourthSecondSmallk} & $1.8080\cdot10^{-5}x^{\frac{39}{44}}(\log{x})^{\frac{105}{11}}$ & $1.2354\cdot10^{71}$ \\ \hline
 $(9,0.348,0.349,10^{72})$ &  \ref{rmk:FourthSecondSmallk} & $1.7754\cdot10^{-5}x^{\frac{1949}{2200}}(\log{x})^{\frac{105}{11}}$ & $1.6000\cdot10^{72}$ \\ \hline
\end{longtable}
\captionof{table}{Results using the fourth moment estimates in Section \ref{sec:EffectiveFourth} if $x \geq x_1$ is an integer or a half-integer and $k\geq 5$.}\label{fourthmomentk5678}
\end{center}

Let us now consider the case $k=4$ with fourth moments. Now neither the exponent of $x$ nor the exponent of logarithm depend on $\varepsilon_1$ -- they are $3/5$ and $5$, respectively. The result $|\Delta_4(x)| \leq 4.48x^{3/4}\log{x}$ by \cite{MR4311680} is the best estimate for $2 \leq x <3.330\cdot10^{175}$. Hence, we first want to improve that estimate for as small $x$ as possible. These leads to the estimates in Table \ref{fourthmomentk4}. We notice that the first result improves \cite[Theorem 1]{MR4311680} for $x \geq x_1$ and the second one for $x\geq 1.264\cdot10^{43}$. In addition, they both improve \cite[Theorem 1.2]{Tudzi2025} throughout the region $x \geq x_1$. Thus, we do not try to improve \cite[Theorem 1.2]{Tudzi2025} with any other choices of parameters.
\begin{center}
\begin{longtable}{ |c|c|c|c| } 
 \hline
 $(k,\varepsilon,\varepsilon_1,x_0,a)$ & Remark & $|\Delta_4(x)| <$... & $x_1$ \\ \hline
 $(4,0.360,0.361,4.0772\cdot10^{40},7)$ & \ref{rmk:FourthFirstSmallk} & $0.0839x^{\frac{3}{5}}(\log{x})^{5}$ & $2.8541\cdot10^{41}$ \\ \hline
 $(4,0.360,0.361,4.0772\cdot10^{40},3)$ & \ref{rmk:FourthSecondSmallk} & $0.1348x^{\frac{3}{5}}(\log{x})^{5}$ & $1.2232\cdot10^{41}$ \\ \hline
 \caption{Results using the fourth moment estimates in Section \ref{sec:EffectiveFourth} if $x \geq x_1$ is an integer or a half-integer and $k=4$.}\label{fourthmomentk4}
\end{longtable}
\end{center}

Next, we derive estimates using the second moment, based on results in Section \ref{sec:effectiveSecondMoment}. For $k\geq 3$, we derive estimates in three different cases similarly as we did with the estimates using the fourth moment for $k\geq 5$. Indeed, we first choose
\begin{equation*}
    \frac{\varepsilon_1(k-2)}{k+4}=\frac{1}{2}\left(\frac{k-1}{k}-\frac{k+1}{k+4}\right) \quad\iff\quad \varepsilon_1=\frac{1}{k}
\end{equation*}
and $\varepsilon=0.9/k$. Then, we find lower bounds for $x_1$. At the end, we aim to improve \cite[Theorem 1.1, Theorem 1.2, Table 1,]{Tudzi2025}. In all of these, we choose $a=1.6$, and only the results where $x_1<10^{100}$ are presented. In the cases $k=5$ and $k=6$ our results would improve \cite[Theorem 1.2, Table 1]{Tudzi2025} for all $x \geq x_1$ only if $x_1>x^{100}$. Hence, we have not included those lower bounds $x_1$ to Table \ref{secondhmomentk3456}. 

\begin{table}
\centering
\begin{threeparttable}
 \caption{Results using the second moment estimates in Section \ref{sec:effectiveSecondMoment} if $x \geq x_1$ is an integer or a half-integer and $k=3, 4,5$ and $6$.}\label{secondhmomentk3456}
\begin{tabular}{ |c|c|c|c| } 
 \hline
 $(k,\varepsilon,\varepsilon_1,x_0)$ & Theorem & $|\Delta_k(x)| <$... & $x_1$ \\ \hline
 $(3,0.3,1/3,2.5832\cdot10^{87})$ & \ref{thm:EffectiveSecond} & $0.2797x^{\frac{13}{21}}(\log{x})^{3}$\tnote{*} & $4.1332\cdot10^{87}$ \\ \hline
 $(3,0.3,1/3,2.5832\cdot10^{87})$ & \ref{thm:EffectiveSecond2} & $181.1755x^{\frac{13}{21}}(\log{x})^{3}$\tnote{*} & $4.1332\cdot10^{87}$ \\ \hline
 $(3,0.649,0.65,\num{110468})$ & \ref{thm:EffectiveSecond} & $3.6127x^{\frac{93}{140}}(\log{x})^{3}$ & $\num{176749}$ \\ \hline
 $(3,0.649,0.65,\num{110468})$ & \ref{thm:EffectiveSecond2} & $744585.9864x^{\frac{93}{140}}(\log{x})^{3}$ & $3.2663\cdot10^{46}$ \\ \hline
  $(3,0.294,0.295,10^{98})$ & \ref{thm:EffectiveSecond} & $0.2475x^{\frac{859}{1400}}(\log{x})^{3}$\tnote{**} & $1.6000\cdot10^{98}$ \\ \hline
 $(4,0.498,0.499,3.5298\cdot10^{11})$ & \ref{thm:EffectiveSecond} & $2.1207x^{\frac{2999}{4000}}(\log{x})^{4}$ & $5.6477\cdot10^{11}$ \\ \hline
 $(4,0.498,0.499,3.5298\cdot10^{11})$ & \ref{thm:EffectiveSecond2} & $1010.5609x^{\frac{2999}{4000}}(\log{x})^{4}$ & $2.0140\cdot 10^{58}$ \\ \hline
 $(4,0.304,0.305,10^{94})$ & \ref{thm:EffectiveSecond2} & $0.0168x^{\frac{561}{800}}(\log{x})^{4}$\tnote{***} & $1.6000\cdot 10^{94}$\\ \hline
 $(5,0.398,0.399,8.8503\cdot10^{28})$ & \ref{thm:EffectiveSecond} & $1.0586x^{\frac{2399}{3000}}(\log{x})^{5}$ & $1.4161\cdot10^{29}$ \\ \hline
 $(5,0.398,0.399,8.8503\cdot10^{28})$ & \ref{thm:EffectiveSecond2} & $4.1503x^{\frac{2399}{3000}}(\log{x})^{5}$ & $7.0277\cdot10^{69}$ \\ \hline
 $(6,0.332,0.333,2.2169\cdot10^{73})$ & \ref{thm:EffectiveSecond} & $0.6040x^{\frac{1041}{1250}}(\log{x})^{6}$ & $3.5471\cdot10^{73}$ \\ \hline
 $(6,0.332,0.333,2.2169\cdot10^{73})$ & \ref{thm:EffectiveSecond2} & $0.0333x^{\frac{1041}{1250}}(\log{x})^{6}$\tnote{****} & $8.2243\cdot10^{80}$ \\ \hline
\end{tabular}
  \begin{tablenotes}[para]
    \item[*] In this case $x_1$ is very large and hence the last term in $r_1$, $\frac{(c-\varepsilon)(1+c-\varepsilon)(2+c-\varepsilon)}{720(ax_1)^{c-\varepsilon+3}}$, is so small that precision may be lost in the computations. Hence, we have estimated it to be $<\frac{(c-\varepsilon)(1+c-\varepsilon)(2+c-\varepsilon)}{720(ax_1)^{c-\varepsilon}}<10^{-92}.$
    \item[**] We estimate $\frac{(c-\varepsilon)(1+c-\varepsilon)(2+c-\varepsilon)}{720(ax_1)^{c-\varepsilon+3}}<\frac{(c-\varepsilon)(1+c-\varepsilon)(2+c-\varepsilon)}{720(ax_1)^{c-\varepsilon}}<10^{-100}.$
    \item[***] We estimate $\frac{(c-\varepsilon)(1+c-\varepsilon)(2+c-\varepsilon)}{720(ax_1)^{c-\varepsilon+3}}<\frac{(c-\varepsilon)(1+c-\varepsilon)(2+c-\varepsilon)}{720(ax_1)^{c-\varepsilon}}<10^{-96}.$
    \item[****] We estimate $\frac{(c-\varepsilon)(1+c-\varepsilon)(2+c-\varepsilon)}{720(ax_1)^{c-\varepsilon+3}}<\frac{(c-\varepsilon)(1+c-\varepsilon)(2+c-\varepsilon)}{720(ax_1)^{c-\varepsilon}}<10^{-83}.$
  \end{tablenotes}
\end{threeparttable}
\end{table}

Let us now consider the case $k=2$. In Theorems \ref{thm:EffectiveSecond} and \ref{thm:EffectiveSecond2}, the estimate
\begin{equation*}
    |\Delta_2(x)|>\frac{x^{\frac{1}{2}}(\log{x})^2}{8\pi}.
\end{equation*}
This can improve \cite[Theorem 1.1, Theorem 1.2]{MR2869048} and \cite[p. 120, footnote 6]{MR4500746} only if $x <\exp(\sqrt{0.961\cdot8\pi}) \approx 136$. Since $a \geq 1.6$, we must have
\begin{equation*}
    \varepsilon \geq \frac{1.5379}{\log{(\sqrt{0.961\cdot8\pi})}-\log\log{(1.6)}} \approx 1.031.
\end{equation*}
With these restrictions, $\omega_3>6.5146$ and $\omega_4>6.2428\cdot10^{11}$. These do not improve \cite[Theorem 1.1, Theorem 1.2]{MR2869048} and \cite[p. 120, footnote 6]{MR4500746}, and hence we do not include these estimates in our results.

Lastly, we consider results using our non-moment estimate, indeed, Theorem \ref{thm:nomoments}. In order to improve estimates for $x<10^{100}$, we must have $x_0< 10^{100}/1.6$ and 
\begin{equation*}
    \varepsilon \geq \frac{\lambda(k)}{\log\log{(10^{100}/1.6)}}\geq \frac{1.5379}{\log\log{(10^{100}/1.6)}}\approx 0.283,
\end{equation*}
rounded to three decimals. Hence we also have $\varepsilon_1\geq 0.284$. However, the exponent of $x$ in Theorem \ref{thm:nomoments} is at least
\begin{equation*}
    \frac{1.566(k-1)+0.284(1.566k-1)}{1.566k+1}>\frac{k-1}{k} \quad\text{for } k\geq 2.
\end{equation*}
Hence, this will not improve the best known exponents (and numerical computations reveal that we cannot improve the best known explicit estimates either) for $x<10^{100}$ using Theorem \ref{thm:nomoments}. Therefore, we do not include these estimates in our results. However, we note that the exponent of $x$ in Theorem \ref{thm:nomoments} goes to $(k-1)/(k+1)$ as $(\varepsilon,\varepsilon_1) \to (0,0)$. Hence, it can provide useful estimates for large values of $x$.

Before applying Lemma \ref{lemma:estimateBetweenPoints}, we check which one of our estimates in Tables \ref{fourthmomentk5678}, \ref{fourthmomentk4} and \ref{secondhmomentk3456} is the best one in different ranges of $x$. The results are presented in Table \ref{abstractResults}. Lastly, we apply Lemma \ref{lemma:estimateBetweenPoints} to derive final results.

\begin{center}
\begin{ThreePartTable}
\begin{TableNotes}
    \item[*] Not included in Theorem \ref{thm:main} since other results in this table will improve the best known existing estimates for smaller $x$, and these will not extend the best known results either.
  \end{TableNotes}
\begin{longtable}{ |c|c|c| } 
 \caption{Best results from Tables \ref{fourthmomentk5678}, \ref{fourthmomentk4} and \ref{secondhmomentk3456}.}\label{abstractResults} \\
 \hline
 $k$ & Size of $x$ & $|\Delta_k(x)| <$... \\ \hline
 $3$ & $\num{176749} \leq x <4.1332\cdot10^{87}$ & $3.6127x^{\frac{93}{140}}(\log{x})^{3}$\tnote{*}  \\ \hline
 $3$ & $4.1332\cdot10^{87} \leq x <1.6000\cdot10^{98}$ & $0.2797x^{\frac{13}{21}}(\log{x})^{3}$\tnote{*}  \\ \hline
 $3$ & $x \geq 1.6000\cdot10^{98}$ & $0.2475x^{\frac{859}{1400}}(\log{x})^{3}$  \\ \hline
 $4$ & $5.6477\cdot10^{11} \leq x <1.2232\cdot10^{41}$ & $2.1207x^{\frac{2999}{4000}}(\log{x})^{4}$\tnote{*}  \\ \hline
 $4$ & $1.2232\cdot10^{41} \leq x < 2.8541\cdot10^{41}$ & $0.1348x^{\frac{3}{5}}(\log{x})^{5}$\tnote{*}   \\ \hline
 $4$ & $x \geq 2.8541\cdot10^{41}$ & $0.0839x^{\frac{3}{5}}(\log{x})^{5}$  \\ \hline
 $5$ & $1666\leq x <1.7552\cdot10^7$ & $9.2710x^{\frac{11}{14}}(\log{x})^{\frac{41}{7}}$  \\ \hline
 $5$ & $1.7552\cdot10^7\leq x <1.6000\cdot 10^{29}$ & $0.6449x^{\frac{24}{35}}(\log{x})^{\frac{41}{7}}$  \\ \hline
 $5$ & $x \geq 1.6000\cdot 10^{29}$ & $0.0207x^{\frac{556}{875}}(\log{x})^{\frac{41}{7}}$  \\ \hline
 $6$ & $\num{162726}\leq x <1.6000\cdot10^{39}$ & $0.2737x^{\frac{33}{40}}(\log{x})^{\frac{27}{4}}$  \\ \hline
 $6$ & $1.6000\cdot10^{39}\leq x <1.1492\cdot10^{72}$ & $0.0050x^{\frac{2929}{4000}}(\log{x})^{\frac{27}{4}}$  \\ \hline
 $6$ & $x \geq 1.1492\cdot10^{72}$ & $0.0031x^{\frac{35}{48}}(\log{x})^{\frac{27}{4}}$  \\ \hline
 $7$ & $x \geq 3.1288\cdot10^{20}$ & $0.0036x^{\frac{5}{6}}(\log{x})^{\frac{23}{3}}$  \\ \hline
 $8$ & $x \geq 3.6972\cdot10^{65}$ & $0.0002x^{\frac{43}{50}}(\log{x})^{\frac{43}{5}}$  \\ \hline
 $9$ & $1.2354\cdot10^{71} \leq x <1.6000\cdot10^{72}$ & $1.8080x^{\frac{39}{44}}(\log{x})^{\frac{105}{11}}$\tnote{*}  \\ \hline
 $9$ & $x \geq 1.6000\cdot10^{72}$ & $1.7754x^{\frac{1949}{2200}}(\log{x})^{\frac{105}{11}}$  \\ \hline
\insertTableNotes
\end{longtable}
\end{ThreePartTable}
\end{center}

\section{Discussion}{\label{Disc}}
In this paper, significant parts of our findings are derived using the size and moments of the Riemann zeta function on the critical line. An improvement in any of these bounds or proving current best known non-explicit results explicitly would directly lead to slight improvement in the exponent of $x$ for $\Delta_{k}(x)$. Particularly, sharper estimates for the second and the fourth moment would improve the bounds. Moreover, if one wishes to derive estimates for very large $x$, there are already better explicit estimates available as mentioned in Remark \ref{rkm:sharperRiemann}. As an example, one could replace $\zeta(1/2+it)=O(|t|^{1/6+\varepsilon}$) used in this paper with the best-known upper bound for $\zeta(1/2+it)=O(|t|^{13/84+\varepsilon})$ due to Bourgain \cite{MR2107036}. This yields
\begin{equation*}
    \Delta_{k}(x)=O\left(x^{\frac{13k+16}{13k+58}+\varepsilon}\right) \ \text{and} \  \Delta_{k}(x)=O\left(x^{\frac{13k-10}{13k+32}+\varepsilon}\right) 
\end{equation*} 
for all $k\ge 2$ and  $k\ge 4$ respectively via our method. Note that we take $T=x^{\frac{42}{13k+58}}$ when $k\ge 2$ and $T=x^{\frac{42}{13k+32}}$ when $k\ge 4$. In principle,  bounds for higher moments, such as the twelfth moment due to Heath--Brown \cite{HDR-1978} (see also \cite{palo2025}) could further improve the resulting estimate for $\Delta_{k}(x)$ when $k\ge 12$. For instance, Lemma \ref{lemma:TkIntegral}, Lemma \ref{lemma:dkGeneral}, Lemma \ref{lemma:sigmaChanges}, and any of the lemmas in Section \ref{sec:momentEstimates} provide the following equivalent results for $k\ge 12$ within the context of the current work provided the results by Heath--Brown and Bourgain are applied:
\begin{equation*}
    \Delta_{k}(x,T)=O(x^{1+\varepsilon}T^{-1})+O(x^{1/2}T^{\frac{13k}{84}-1+\varepsilon})+O(x^{1/2}T^{\frac{13(k-12)}{84}+1+\varepsilon}),
\end{equation*}
where $\varepsilon>0$. Taking $T=x^{42/(13k+12)}$, the first and last terms are of the same order, and each dominates the remaining term. Hence, this yields
\begin{equation}
\label{eq:12thmomentDelta}
    \Delta_{k}(x)=O\left(x^{\frac{13k-30}{13k+12}+\varepsilon}\right),
\end{equation}
which represents a significant improvement in the exponent of $x$ over the earlier bounds obtained in this paper. However, explicit versions of these bounds remains out of reach. Hence, one may wish to try completely other methods to derive explicit estimates based on known non-explicit results, see for example \cite{BellottiYang2024}. Theorem 1.1 in \cite{BellottiYang2024} improves \eqref{eq:12thmomentDelta} for all $k \geq 30$.

Conjecturally, similar growth rates for $\Delta_{k}(x)$ are suggested by Keating and Snaith \cite{Keating2000RandomMT}'s predictions from Random Matrix Theory for moments of the zeta function. In particular, if the conjectured bounds for all moments were known, they would imply the Lindel\"of hypothesis, which predicts that $|\zeta(1/2+it)|\ll t^{\varepsilon}$, where $\varepsilon>0$. Under this assumption, we obtain $\Delta_{k}(x)=O(x^{1/2+\varepsilon})$ which is also better than estimate in \cite[Theorem 1.1]{BellottiYang2024}. For further background on the zeta function's moments and their relationship to conjectures like Lindel\"of's, see for example \cite[Chapter 8]{MR1994094} and \cite[Chapters 13 and 14]{MR882550}.

\section*{Acknowledgment}
We would like to thank Shashi Chourasiya, Nicol Leong, Aleksander Simoni\v{c} and Tim Trudgian for useful discussions that improved this manuscript. The research of Neea Paloj\"arvi was supported by the Finnish Cultural Foundation.

\printbibliography
\end{document}